\documentclass[12pt,twoside]{article}
\usepackage{amsmath}
\usepackage{graphicx}
\usepackage{yhmath}
\usepackage{mathdots}
\usepackage[nottoc,numbib]{tocbibind}
\usepackage{mathtools}
\usepackage[all,cmtip]{xy}
 \usepackage{amsfonts}
\usepackage{amsthm}
\usepackage{amsmath,amscd}

\usepackage{amssymb}
\date{June 18, 2021}

\begin{document}

\centerline {\Large{\bf  Symmetric mutation algebras in the context of sub-cluster algebras}}

\centerline{}

\centerline{}

\centerline{\bf {Ibrahim Saleh}}

\centerline{Email: salehi@uww.edu}

 \newtheorem{thm}{Theorem}[section]
 \newtheorem{cor}[thm]{Corollary}
  \newtheorem{cor and defn}[thm]{Corollary and Definition}
 \newtheorem{lem}[thm]{Lemma}
 \newtheorem{prop}[thm]{Proposition}
 \theoremstyle{definition}
 \newtheorem{defn}[thm]{Definition}
 \newtheorem{defns}[thm]{Definitions}
 \theoremstyle{remark}
 \newtheorem{rem}[thm]{Remark}
 \newtheorem{rems}[thm]{Remarks}
 \newtheorem{exam}[thm]{Example}
 \newtheorem{exams}[thm]{Examples}
 \newtheorem{conj}[thm]{Conjecture}
 \newtheorem{que}[thm]{Question}
  \newtheorem{ques}[thm]{Questions}
  \newtheorem{ques and Conj}[thm]{Questions and Conjecture}
 \newtheorem{rem and def}[thm]{Remark and Definition}
 \newtheorem{def and rem}[thm]{Definition and Remark}
 \newtheorem{corr}[thm]{Corollary of the Proof of Proposition 5.15}
 \numberwithin{equation}{section}
\newtheorem{IbrI}{Lemma}[section]
\newtheorem{chiral}[IbrI]{Definition}
\newtheorem{IbrII}[IbrI]{Lemma}
\newcommand{\field}[1]{\mathbb{#1}}
 \newtheorem{mapProof}[thm]{Map of the Proof}
 \newcommand{\eps}{\varepsilon}
 \newcommand{\To}{\longrightarrow}
 \newcommand{\h}{\mathcal{H}}
 \newcommand{\s}{\mathcal{S}}
 \newcommand{\A}{\mathcal{A}}
 \newcommand{\J}{\mathcal{J}}
 \newcommand{\M}{\mathcal{M}}
 \newcommand{\W}{\mathcal{W}}
 \newcommand{\X}{\mathcal{X}}
 \newcommand{\BOP}{\mathbf{B}}
 \newcommand{\BH}{\mathbf{B}(\mathcal{H})}
 \newcommand{\KH}{\mathcal{K}(\mathcal{H})}
 \newcommand{\Real}{\mathbb{R}}
 \newcommand{\Complex}{\mathbb{C}}
 \newcommand{\Field}{\mathbb{F}}
 \newcommand{\RPlus}{\Real^{+}}
 \newcommand{\Polar}{\mathcal{P}_{\s}}
 \newcommand{\Poly}{\mathcal{P}(E)}
 \newcommand{\EssD}{\mathcal{D}}
 \newcommand{\Lom}{\mathcal{L}}
 \newcommand{\States}{\mathcal{T}}
 \newcommand{\abs}[1]{\left\vert#1\right\vert}
 \newcommand{\set}[1]{\left\{#1\right\}}
 \newcommand{\seq}[1]{\left<#1\right>}
 \newcommand{\norm}[1]{\left\Vert#1\right\Vert}
 \newcommand{\essnorm}[1]{\norm{#1}_{\ess}}
 \newcommand{\beq}{\begin{equation}}
\newcommand{\eeq}{\end{equation}}
\newcommand{\rarr}{\rightarrow}
\newcommand{\cA}{\mathcal{A}}
\newcommand{\cS}{\mathcal{S}}
\newcommand{\cC}{\mathcal{C}}
\newcommand{\cU}{\mathcal{U}}
\newcommand{\cR}{\mathcal{R}}
\newcommand{\RMod}{R\text{-Mod}}
\newcommand{\AMod}{A\text{-Mod}}
\newcommand{\Rep}{\text{Rep}}
\newcommand{\Aut}{\text{Aut}}
\newcommand{\XAut}{\xi\Aut}
\newcommand{\Rtzn}{R \{\theta, z, n \}}
\newcommand{\TxC}{(\Theta, \xi, \cC)}
\newcommand{\Chir}{\text{Chir}}

\tableofcontents

\begin{abstract} For a rooted cluster algebra $\mathcal{A}(Q)$ over a valued quiver $Q$,  a \emph{symmetric cluster variable} is any cluster variable  belonging to a cluster  associated with a quiver $\sigma (Q)$, for some permutation $\sigma$. The subalgebra of $\mathcal{A}(Q)$ generated by all symmetric cluster variables is called the \emph{symmetric mutation subalgebra} and is denoted by $\mathcal{B}(Q)$. In this paper we identify the class of cluster algebras that satisfy  $\mathcal{B}(Q)=\mathcal{A}(Q)$, which contains  almost every quiver of  finite mutation type.  In the process of proving the main result, we provide a  classification of quivers mutations classes  that relates their maximum weights to the shapes of the initial quivers. Some properties of  symmetric mutation subalgebras are  given.

\end{abstract}



\section{Introduction}

Cluster algebras were introduced by S. Fomin and A. Zelevinsky in [5, 6, 2, 7, 15]. A cluster algebra is a commutative ring with a distinguished set of generators called cluster variables which appear in overlapping sets called clusters. Each cluster is paired with a (valued) quiver to form what is  called a seed. A new seed can be obtained from an existing  seed using particular type of relations called mutations. The set of all quivers that can be produced from a quiver $Q$  by applying mutations is called \emph{mutation class} of $Q$. The mutation class  plays a central role in understanding the structure of its associated  cluster algebra. In [6], S. Fomin and A. Zelevinsky proved that a cluster algebra has a finite set of seeds  if and only if each quiver in its  mutation class has weight less than or equal to three. Finding sufficient and/or equivalent conditions for the mutation class of a quiver  to be a finite set has always been an important goal. In [4] and [3]  the authors proved that the mutation class is finite if and only if each quiver in the mutation class is of weight less than or equal to four. Finite mutation classes have been studied and some of their characterizations have been revealed in many papers after that, including [12].  In [11, 9, 8, 12], the notions of mutation groups,  subseeds and subcluster algebras were introduced and studied.

Understanding  the ring-theoretical  structure of cluster algebras has not  received the attention it deserves. An approach toward  this direction is through studying the mutations groups and their possible relations to  the subalgebras of the associated cluster algebras. We introduce what we called \emph{mutations subalgebras}, which are subalgberas of the cluster algebra where each one is created by a subgroup of the mutation group. That is, each subgroup $H$ of the mutation group, generates a subset of cluster variables $\mathcal{X}_{H}$, where the algebra  generated by $\mathcal{X}_{H}$ is called \emph{mutation subalgebra} associated to $H$.

 Two seeds $(X, Q)$ and $(Y, Q')$ are said to be \emph{symmetric} seeds if the quivers $Q$ and $Q'$ are symmetric, i.e.,  there is a permutation $\sigma$ such that $Q'=\sigma(Q)$, see [13] under $\sigma$-similar. The subgroup of the mutation group, generated by all sequences of mutations which produce seeds that are symmetric to the initial seed, is called the \emph{symmetric mutation subgroup}. The corresponding mutation subalgebra is called the \emph{ symmetric mutation subalgebra}. The main purpose of this paper is  to answer the following question.
\begin{que}
For which quivers, the symmetric mutation subalgebra coincides with the whole cluster algebra?
\end{que}
  We provide a precise answer for this question, by proving that ``the class of valued quivers where  every cluster variable is symmetric, is  identified with almost the whole  class of  finite mutation quivers", see Theorem 3.19. The main motive of this question  was trying to use the mutations classes to help understand the ``algebraic" structure of the associated cluster algebra.  The paper has the following two other aims.
   \begin{enumerate}
     \item [A.] Without the need to obtain the whole mutation class of a valued quiver $Q$, can we predict whether $Q$ is of finite type or not?  see Lemma 3.11.
     \item [B.] Exploring particular types of generalized  ``local" non-trivial \emph{inverse mutation} sequences.  For  details see  Proposition 3.15 and Lemma 3.18. 
   \end{enumerate}

All the proofs  are written using the  valued quivers version of the  definition of skew-summarizable matrices mutation, see Definition 2.2 and Remarks 2.3.

The paper is organized as follows: In the second section we provide a brief introduction to valued quivers and their rooted cluster algebras. The third section contains an introduction to symmetric mutation subalgebra and a statement that characterizes  mutations classes based on their weights and the shape of the initial quiver, Lemma 3.11. In the same section we also prove Theorem 3.19 which provides an answer for Question 1.1 by giving a classification for quivers which all their cluster variables are symmetric. Finally we provide a statement that gives some elementary properties about the structure of symmetric mutation subalgebras.

 Throughout the paper, $K$ is a field of zero characteristic and the notation $[1, k]$ stands for the set $\{1,\ldots, k\}$. All quivers are valued quivers and all cluster algebras are rooted. The filed of rational functions in $n$ independent indeterminists over $K$ is given by $\mathcal{F}$.


\section{Rooted Cluster Algebras}

\begin{defns}
\begin{enumerate}
  \item \emph{An oriented valued quiver} of rank $n$ is a quadruple $Q=(Q_{0}, Q_{1}, V, d)$, where
  \begin{itemize}
 \item $Q_{0}$ is a  set of $n$ vertices labeled by $[1,n]$.
  \item  $Q_{1}$ is a set of ordered pairs of vertices, that is $Q_{1}\subset Q_{0}\times Q_{0}$ such that; $(i, i)\notin Q_{1}$ for every $i\in Q_{0}$, and if $(i, j)\in Q_{1}$, then $(j, i)\notin Q_{1}$.
  \item $V=\{(v_{ij},v_{ji})\in \mathbb{N}\times\mathbb{N} | (i,j)\in Q_{1}\}$, $V$ is called the valuation of $Q$. The weight of an edge $\alpha=(i, j)$ is the product $v_{ij}v_{ji}$ and will be denoted by $w(\alpha)$ or $w_{i,j}$.
  \item $d=(d_{1},\cdots, d_{n})$ where $d_{i}$ is a positive integer for each $i$, such that $d_{i}v_{ij}=v_{ji}d_{j}$ for every $i, j\in [1, n]$. This condition will be called the consistency condition.
  \item
  In the case of $(i,j)\in Q_{1}$, then there is an arrow oriented from $i$ to $j$ and in notation we shall use the symbol $\xymatrix{{\cdot}_{i} \ar[r]^{(v_{ij},v_{ji})}&{\cdot}_{j}}$ and if  the emphasis is on the weight of the edge we use $\xymatrix{{\cdot}_{i} \ar[r]^{w_{i,j}}&{\cdot}_{j}}$.
  \item
  We will also use $rk(Q)$ for rank of $Q$.
   \item
   If $w_{i, j}\leq 1$ for every two vertices $i$ and $j$ in $Q$, then we call $Q$ a \emph{simply-laced quiver}.
    \item
    An oriented $3$-cycle that has two edges with the same weights will be called \emph{isosceles $3$-cycle} and if all three edges are with equal weights  then the  $3$-cycle will be called an \emph{equilateral $3$-cycle.}
 \end{itemize}
\item For $i\in Q_{0}$, we define $Nhb.(i)=\{j \in Q_{0}; w_{i,j}\neq 0 \}$.

\item
  A quiver $Q$ is called a \emph{zigzag quiver}, if  every vertex $i\in Q_{0}$, is one of the following three cases

   \begin{enumerate}
   \item  a leaf, i.e., $Nhb.(i)$ contains exactly one vertex,
     \item a source, i.e.,  $Nhb.(i)$ has the form   $\xymatrix{{\cdot} &\ar[l]{\cdot}_{i} \ar[r]&{\cdot}}$, or
     \item a target, i.e.,   $Nhb.(i)$ has the form $\xymatrix{{\cdot}\ar[r] &{\cdot}_{i} &\ar[l]{\cdot}}$.
   \end{enumerate}

\item We use $-Q$ for the valued quiver obtained from $Q$ by reversing all the arrows and valuations.

  \item  A valued quiver $Q'$ is called \emph{symmetric} to $Q$ if  there is a permutation $\tau$ such that $Q'$ can be obtained from $Q$ by permuting the vertices of $Q$ using $\tau$ such that for every edge $\xymatrix{{\cdot}_{i} \ar[r]&{\cdot}_{j}}$ in $Q$,   the valuation $(v_{ij}, v_{ji})$ is assigned  to the edge $\xymatrix{{\cdot}_{\tau(i)} \ar[r]&{\cdot}_{\tau(j)}}$ in $Q'$.

  \item \emph{A seed } in $\mathcal{F}$ of rank $n$  is a pair $(\widetilde{X}, \widetilde{Q})$, where
   \begin{enumerate}
     \item The $n+m$-tuple $\widetilde{X}=(x_{1},\ldots, x_{n}, x_{n+1}, \ldots, x_{n+m})$ is called an \emph{extended cluster} such that
 $X=(x_{1}, x_{2},\ldots,x_{n})\in \mathcal{F}^n$ is a transcendence basis of of $\mathcal{F}$ over $K[x_{n+1},\ldots, x_{n+m}]$  that generates $\mathcal{F}$ which is called a \emph{cluster}. Elements of $X$ will be called \emph{cluster variables} and the elements $\{x_{n+1},\ldots, x_{n+m}\}$ will be called \emph{frozen variables};
     \item $\widetilde{Q}$ is an oriented valued quiver with $n+m$ vertices such that it has a connected subquiver $Q$ of rank $n$. The vertices of  $Q$ are labeled by numbers from $[1, n]$ which will be called exchangeable vertices;

     \item Each element of $\widetilde{X}$ is assigned to a vertex in $\widetilde{Q}$ such that the cluster variables are assigned to the vertices of the subquiver $Q$.
   \end{enumerate}

\end{enumerate}
\end{defns}
We note that every oriented valued quiver $\widetilde{Q}=(Q_{0}, Q_{1}, V, d)$ corresponds to a skew symmetrizable matrix   $B(\widetilde{Q})=(b_{ij})$ given by
\begin{equation}\label{}
  b_{ij}=\begin{cases} v_{ij}, & \text{ if }(i,j)\in Q_{1},\\
    0, & \text{ if }i=j,\\
-v_{ij}, & \text{ if }(j,i)\in Q_{1}.
    \end{cases}
\end{equation}
One can also see that; every skew symmetrizable matrix $B$ corresponds to an oriented valued quiver $\widetilde{Q}$ such that $B(\widetilde{Q})=B$, see [10] for more details.

 All our valued quivers are oriented, so in the rest of the paper  we will omit the word ``oriented". If no confusion, from time to time  we will use $Q$ for $\widetilde{Q}$ and $X$ for $\widetilde{X}$. All quivers are of rank $n$ unless stated otherwise.

   \begin{defn}[\emph{Valued quiver mutations}]
 Let $\widetilde{Q}$ be a valued quiver. The mutation  $\mu_{k}(\widetilde{Q})$ at a vertex  $k$  is defined  through Fomin-Zelevinsky's mutation of the associated skew-symmetrizable matrix. The mutation of a skew symmetrizable matrix $B=(b_{ij})$ on the direction $k\in [1, n]$ is given by $\mu_{k}(B)=(b'_{ij})$, where
\begin{equation}
b'_{ij}=\begin{cases} -b_{ij}, & \text{if} \ k \in \{i,j\},\\
   b_{ij}+\text{sign}(b_{ik})\max(0, b_{ik}b_{kj}), & \text{otherwise.}
   \end{cases}
  \end{equation}

\end{defn}
The following remarks  provide an  adequate set of  rules to  calculate  mutations of  valued quivers without using their associated skew-symmetrizable matrix.

\begin{rems}
\begin{enumerate}

  \item  Let $\widetilde{Q}=(Q_{0}, Q_{1}, V, d)$ be a valued quiver. The mutation  $\mu _{k}(\widetilde{Q})=(Q_{0}, Q'_{1}, V', d)$ at the  vertex  $k$   is described using the mutation of $B(\widetilde{Q})$ as follows: we obtain $Q'_{1}$ and $V'$,   by altering  $Q_{1}$ and $V$, based on the following rules
\begin{enumerate}
 \item replace the  pairs $(i, k)$ and $(k,j)$  with $(k,i)$ and $(j,k)$  respectively and switch the components of the ordered pairs of their valuations;
  \item if  $(i,k), (k,j)\in Q_{1}$, such that neither of $(j,i)$  or $(i,j)$ is in $Q_{1}$ (respectively $(i,j)\in Q_{1}$) add the  pair $(i, j)$ to $Q'_{1}$, and give it the valuation $(v_{ik}v_{kj},v_{ki}v_{jk})$ (respectively change its valuation to $(v_{ij}+v_{ik}v_{kj},v_{ji}+v_{ki}v_{jk})$);
 \item if $(i,k)$, $(k,j)$ and $(j, i)$ in $Q_{1}$, then we have three cases
 \begin{enumerate}
   \item if $v_{ik}v_{kj}<v_{ij}$, then keep $(j,i)$ and change its valuation to $(v_{ji}-v_{jk}v_{ki}, -v_{ij}+v_{ik}v_{kj})$;
   \item if $v_{ik}v_{kj}>v_{ij}$, then replace $(j,i)$ with $(i,j)$ and change its valuation to $( -v_{ij}+v_{ik}v_{kj}, |v_{ji}-v_{jk}v_{ki}|)$;
   \item if $v_{ik}v_{kj}=v_{ij}$,  then remove $(j,i)$ and its valuation.
 \end{enumerate}
 \item $d$ will stay the same in $\mu _{k}(\widetilde{Q})$.
\end{enumerate}
 \item The mutation of valued quiver is again a valued quiver.
 \item One can see that; $\mu^{2}_{k}(\widetilde{Q})=\widetilde{Q}$ and  $\mu_{k}(B(\widetilde{Q}))=B(\mu_{k}(\widetilde{Q}))$ at each vertex  $k\in [1, n]$ where $\mu_{k}(B(\widetilde{Q}))$ is the mutation on the matrix $B(\widetilde{Q})$. For more information on mutations of skew-symmetrizable matrices see for example [6, 15, 10].
 \end{enumerate}
\end{rems}

\begin{defn}[Seed mutation] Let $(\widetilde{X}, \widetilde{Q})$ be a seed in $\mathcal{F}$. For each fixed $k\in [1,n]$,
 we define a new seed $\mu_{k}(\widetilde{X},\widetilde{Q})=(\mu _{k}(\widetilde{X}), \mu _{k}(\widetilde{Q}))$ by  setting $\mu _{k}(\widetilde{X})=(x'_{1}, \ldots, x'_{n}, x_{n+1}, \ldots, x_{n+m})$  where
\begin{equation}\label{}
   x'_{i}=\begin{cases} x_{i}, & \text{ if } i\neq k,\\
    \frac{ \prod\limits_{b_{ji}> 0} x_{j}^{b_{ji}}+
    \prod\limits_{b_{ji}< 0} x_{j}^{-b_{ji}}}{x_{i}}, & \text{ if }i=k.
    \end{cases}
\end{equation}

And $\mu _{k}(\widetilde{Q})$  is the mutation of $\widetilde{Q}$ at the exchangable vertex $k\in [1, n]$.

   \end{defn}

\begin{defns}[Cluster structure and cluster algebra]
\begin{itemize}
  \item The set of all seeds obtained by applying all possible sequences of mutations on the seed $(\widetilde{X}, \widetilde{Q})$ is called the  \emph{cluster structure} of $(\widetilde{X}, \widetilde{Q})$.
       \item
      The set of all quivers appear in the cluster structure of $(\widetilde{X}, \widetilde{Q})$ is called the \emph{mutation class} of $Q$ and will be denoted by $[Q]$.

        \item Let $\mathcal{X}$ be the union of all clusters in the cluster structure of $(\widetilde{X}, \widetilde{Q})$. The  \emph{rooted cluster algebra} $\mathcal{A}(Q)$ is the $\mathbb{Z}[x_{n+1}, \ldots, x_{n+m}]$-subalgebra of $\mathcal{F}$ generated by $\mathcal{X}$. For simplicity we will omit the word ``rooted".
      \item  Let $\mathcal{S}_{n}$ be the symmetric group in $n$ letters. One can introduce an action of  $\mathcal{S}_{n}$ in the set of quivers of rank $n$ as follows: for a permutation $\tau$, the quiver $\tau(Q)$  is obtained from $Q$ by permuting the vertices of $Q$ using $\tau$ such that for every edge $\xymatrix{{\cdot}_{i} \ar[r]&{\cdot}_{j}}$ in $Q$,   the valuation $(v_{ij}, v_{ji})$ is assigned  to the edge $\xymatrix{{\cdot}_{\tau(i)} \ar[r]&{\cdot}_{\tau(j)}}$ in $\tau(Q)$.

\end{itemize}
\end{defns}
One can see that any seed in the cluster structure of  $(\widetilde{X}, \widetilde{Q})$ generates the same cluster structure.

\begin{defn}(\textbf{Cluster pattern of $\mathcal{A}(Q)$} [7]).
The cluster pattern $\mathbb{T}_{n}(Q)$ of the cluster algebra $\mathcal{A}(Q)$ is a regular $n-$ary tree whose edges are labeled by the numbers $1,2, \ldots, n$ such that the $n$ edges emanating from each vertex receive different labels. The vertices are assigned to be the elements of the cluster structure of  $\mathcal{A}(Q)$ such that the endpoints of any edge are obtained from each other by seed mutation in the direction of the edge label.
\end{defn}
One can see that the cluster pattern of $\mathcal{A}(Q)$ can be completely determined by any seed in the cluster structure.
\begin{defn}
A cluster algebra, $\mathcal{A}(Q)$, is called of \emph{finite type},  if the set of all cluster variables $\mathcal{X}$ of $Q$ is finite. A quiver $Q$ is called of \emph{finite mutation type} if the mutation class $[Q]$ contains finitely many quivers.
\end{defn}

     \begin{exams}
     \begin{enumerate}
       \item Let
     \begin{equation}\label{}
     \widetilde{Q}=\xymatrix{ \cdot_{3_{1}}& \cdot_{3}   \ar[l]_{(2,3)}\ar[r]^{(2,3)}&   \cdot_{2}\ar[d]^{(1,2)}&\cdot_{2_{1}.}\ar[l]_{(2,1)}\\
  &\cdot_{1_{1}} \ar[r]&\cdot_{1}\ar[ul]^{(6,2)}\ar[r]^{(2,3)}&\cdot_{1_{2}} }
     \end{equation}
     Here, the subquiver, with the exchangeable vertices,  $Q$ is the quiver formed from the vertices   $1, 2$ and $3$ with $d=(1, 2, 3)$. So $rk(\widetilde{Q})=3$.
      Applying mutation at the vertex  $2$, produces the following  quiver

  \begin{equation}\label{}
      \nonumber  \mu_{2}(\widetilde{Q})=\xymatrix{ \cdot^{3_{1}}& \cdot_{3}   \ar[l]_{(2,3)}&   \cdot_{2} \ar[l]_{(3,2)}\ar[r]^{(1,2)}&\cdot_{2_{1}}\ar[dl]^{(2, 2)}\\
  &\cdot_{1_{1}} \ar[r]&\cdot_{1}\ar[u]^{(2,1)}\ar[r]_{(2,3)}&\cdot_{1_{2}}.}
     \end{equation}

\item Consider the quiver
       $Q=\xymatrix{ \cdot_{n}& \cdot_{n-1}\ar[l]& \cdots \ar[l]& \cdot_{2}\ar[l]&\cdot_{1}\ar[l]_{(2, 1)} }$. One can see that
       $\mu_{n}\mu_{n-1}\cdots\mu_{3}\mu_{2}(Q)=\xymatrix{ \cdot_{1}& \cdot_{n}\ar[l]_{(1, 2)}& \cdots \ar[l]& \cdot_{3}\ar[l]&\cdot_{2}\ar[l] }$.

\item The following  example is for a simply-laced quiver, yet it has an infinite mutation class
\begin{equation}
 \nonumber   \xymatrix{\cdot_{k}\ar@{-}[dr]\ar[dr]\\
 \cdot_{v}  \ar[u]\ar[dr]\ar[r]& \cdot_{j} \\
 & \ar[u]\cdot_{i}} \  \ \bigskip
\xRightarrow{\mu_{i}\mu_{k}} \ \xymatrix{\cdot_{k}\ar@{-}[dr]\ar[dr]\\
 \cdot_{j}  \ar[u]\ar[dr]\ar[r]^{(3,3)}& \cdot_{v} \\
 & \ar[u]\cdot_{i}}
\end{equation}

     \end{enumerate}
  \end{exams}
\begin{defns}[Subseeds and subcluster algebras]
\begin{enumerate}

  \item  Let $\widetilde{Q}$ be a quiver of rank $n$, with total $m$  vertices and $\widetilde{X}=F\cup X$ be its extended cluster  with   a set of  frozen variables $F$ that has $m-n$ elements. For a set $I\subseteq [1, n]$,  a pair $(\widetilde{Y}, \widetilde{Q}_{I})$ is obtained from the seed $(\widetilde{X}, \widetilde{Q})$ by converting the set of the cluster variables, labeled by the vertices of  $I$, into frozen variables. Where $\widetilde{Y}=F'\cup X'$ where $F'=F\cup I$  and $X'=X\backslash I$ are  the set of frozen variables and the cluster of  $(\widetilde{Y}, \widetilde{Q}_{I})$ respectively. The pair $(\widetilde{Y}, \widetilde{Q}_{I})$ is called  a \emph{subseed} of $(\widetilde{X}, \widetilde{Q})$. Here, $(\widetilde{Y}, \widetilde{Q}_{I})$ is of rank $l=n-|I|$ as a seed and $\widetilde{Q}_{I}=\widetilde{Q}$ as a quiver. We will use $\widetilde{Q}_{I}\leq \widetilde{Q}$ to say $\widetilde{Q}_{I}$ is a subquiver of $\widetilde{Q}$. The exchangeable vertices of $Q_{I}$ are labeled by  $[1, n]\backslash I$.

      \item
      A subquiver $Q_{I}$ is said to be of $A$-type if its mutation class  $[Q_{I}]$ contains a quiver whose exchangeable vertices, $[1, n]\backslash I$, form a subquiver   with underlying  graph  of $A_{n}$-type  such that $w[Q_{I}]=1$.

  \item The cluster algebra $\mathcal{A}(Q_{I})$ of $(\widetilde{Y}, \widetilde{Q}_{I})$ is called a \emph{subcluster algebra} of $\mathcal{A}(Q)$.

  \item We say that the quiver $\widetilde{Q}$ can be \emph{decomposed coherently} into two subquivers $\hat{Q}$ and $\breve{Q}$, if $\widetilde{Q}$ can be formed from $\hat{Q}$ and $\breve{Q}$ by connecting them at one or more (overlapping) vertices. In such case we write $\widetilde{Q}=\hat{Q}\odot\breve{Q}$. In other words, if $\widetilde{Q}=\hat{Q}\odot\breve{Q}$ then $\hat{Q}=\widetilde{Q}_{\breve{Q}_{0}}$ and  $\breve{Q}=\widetilde{Q}_{\hat{Q}_{0}}$, where $\breve{Q}_{0}$ and $ \hat{Q}_{0}$ are the set of vertices in $\breve{Q}$ and $\hat{Q}$ respectively.
\end{enumerate}

\begin{exam} Consider the Quiver (2.4) in Example 2.8. We have $\widetilde{Q}=\hat{Q}\odot\breve{Q}$, where  $\hat{Q}_{0}=\{3_{1}, 3, 2, 2_{1}\}$ and $\breve{Q}_{0}=\{1_{1},1, 1_{2}\}$. One can see that the coherent composition of $\widetilde{Q}$  is not unique.
\end{exam}

\end{defns}

\section{Symmetric mutation sub-algebras}

\begin{defns}
\begin{enumerate}
\item
Fix a quiver $Q$ of rank $n$. Consider the set $\mathcal{M}(Q)$ of all reduced words formed from the single mutations $\mu_{1}, \ldots, \mu_{n}$ on $Q$. The group generated by the elements of $\mathcal{M}(Q)$ will  be called \emph{the mutation group} of $Q$ and it will also be denoted  by $\mathcal{M}(Q)$. The group relations of $\mathcal{M}(Q)$ are due to the applications of the mutations sequences on $Q$.
\item
An element $\mu \in \mathcal{M}(Q)$ is called a \emph{symmetric mutation sequence} on $Q$ if $\mu(Q)=\tau (Q)$ for some permutation $\tau$. The subgroup of $\mathcal{M}(Q)$ generated by all symmetric mutations on $Q$ will be called the \emph{symmetric mutation subgroup} of $Q$ and will be denoted by $\mathbf{B}(Q)$.

\item  Let $\mathcal{X}_{\mathbf{B}}$ be the set of all cluster variables generated by applying the elements of $\mathbf{B}(Q)$  on the seed $(\widetilde{X}, \widetilde{Q})$. The elements of  $\mathcal{X}_{\mathbf{B}}$ will be called \emph{symmetric cluster variables}.
    The subalgebra $\mathcal{B}(Q)$ generated by $\mathcal{X}_{\mathbf{B}}$ over $\mathbb{Z}[x_{n+1}, \ldots, x_{m+n}]$ will be called the \emph{symmetric mutation subalgebra} of $\mathcal{A}(Q)$. Moreover, The cluster algebra $\mathcal{A}(Q)$ will be called \emph{symmetric }if and only if $\mathcal{A}(Q)=\mathcal{B}(Q)$.
\end{enumerate}
\end{defns}

\begin{ques}
\begin{enumerate}

\item Find sufficient conditions  on a quiver $Q$  where $\mathcal{B}(Q)$ is a subcluster algebra of $\mathcal{A}(Q)$, i.e., there is a subquiver $Q'$ of $Q$ such that $\mathcal{A}(Q')=\mathcal{B}(Q)$?
\item
Are there any significant  relations between the subgroup $\mathbf{B}(Q)$ and the set of all elements of  the group $\mathcal{M}(Q)$ with finite order?

\end{enumerate}
\end{ques}

We provide some partial answer for the first question in this paper by identifying sufficient and equivalent conditions for the cluster algebra $\mathcal{A}(Q)$ to be symmetric, i.e., $\mathcal{A}(Q)=\mathcal{B}(Q)$.
\begin{exams} \begin{enumerate}
\item Each of the following quivers generates a symmetric cluster algebra

\begin{equation}
 \nonumber \  \ (a) Q: \xymatrix{
	\cdot_{i}  & \ar[l]\cdot_{j} }  \ \  \  \  \ \ \  \  \  \ \ \ \ \ \ \ \ \ \ \  \ \ \ \ \  \ \ \ \ \ \ \ \ \ \ \ \ \ \ \ \ \
  (b)Q: \xymatrix{
	\cdot_{i} \ar[d]_{(2, 2)} & \ar[l]_{(2, 2)}\cdot_{j} \\
	\cdot_{k}  \ar[ur]_{(2. 2)}  }
 \end{equation}
 \item
 The following quiver satisfies that $\mathcal{B}(Q)=\mathcal{A}(Q')$, where $\widetilde{Q}'=\widetilde{Q}_{I}$ and  $I=\{i, j, k\}$

\begin{equation}
 \  \xymatrix{\cdot_{2}\ar@{-}[r]&\cdot_{1} \\
\cdot_{i}\ar@{-}[ru] \ar[d]_{(3, 3)} & \ar[l]_{(3,2)}\cdot_{k} \\
	\cdot_{j}  \ar[ur]_{(2. 3)}}
 \end{equation}

\item\textbf{\emph{ Rigid Vertices}}. The following quivers satisfy that $\mathcal{B}(Q)\subsetneq\mathcal{A}(Q)$

\begin{equation}
 (a)    \  \xymatrix{\cdot_{2}\ar@{-}[r]^{z}&\cdot_{1} \\
\cdot_{\textbf{i}}\ar@{-}[ru]^{y} \ar[d]_{2} & \ar[l]_{2}\cdot_{k} \\
	\cdot_{j}  \ar[ur]_{(2. 2)}}  \  \  \  \ \ \  \  \  \ \ \ \ \ \ \ \ \ \ \  \ \ \ \ \  \ \ \ \ \ \ \ \ \ \ \ \ \ \ \ \ \ \ \ \
 (b)  \  \xymatrix{\cdot_{k}\ar@{-}[dr]\ar[dr]\\
 \cdot_{v}  \ar[u]\ar[dr]_{2}& \cdot_{j}\ar[l]_{(2, 2)} \\
 & \ar[u]_{2}\cdot_{\textbf{i}}}
\end{equation}
 Where $(y, z)\in \{(1,0), (1, 1), (0, 0)\}$. A good exercise for the reader is to show that $\mu_{\textbf{i}}(x_{\textbf{i}})$  is not a  symmetric  cluster variable.

 \textbf{\emph{ Definition of Rigid Vertices.}}  A quiver $Q$ is said to have a \emph{rigid vertex}, if it contains a ``connected"  subquiver that is symmetric to any one of the four quivers in (3.2)  and we will refer to such vertex, $\textbf{i}$,  by a \emph{\textbf{rigid vertex}} of $Q$.

One can see that a  quiver $Q$ could have more than one rigid vertex.
\end{enumerate}
\end{exams}

The main purpose of the paper is to show that, a quiver $Q$ has a non-symmetric cluster variables if and only if $Q$ contains rigid vertices. 

\begin{defns}
\begin{enumerate}

\item We define \emph{the  weight of} a quiver $Q$, denoted by $w(Q)$,  to be $max\{w(\alpha); \alpha \ \text{is an edge in} \ Q \}$ the maximum weight of the edges of $Q$. If the class $[Q]$ is finite  then  the number $w[Q]=max\{w(Q'); Q'\in [Q]\}$  will be called  the \emph{weight of the mutation class} $[Q]$.
\item An edge $[ij], \xymatrix{\cdot_{i}\ar@{-}[r]^{w_{i,j}}&\cdot_{j}} $ in $Q$ is called a \textit{heavy weight} if $w_{i,j} \geq 5$, and it will be called a \emph{blocking edge} if the permutation  $(ij)$ is not in $\mathcal{M}(Q)$, i.e., there is \textbf{no} sequence of mutations $\mu \in \mathcal{M}(Q)$ such that $(ij)(Q)=\mu(Q)$.

\item A subgraph $P(Q)$ of  a rooted cluster pattern $\mathbb{T}(Q)$   is called a \emph{rooted path} if  it contains the vertex that is associated to the initial seed $(X, Q)$.

		\item A quiver $Q$ is said to be  \emph{unbounded} if it has a $3$-cycle  sub-quiver $Q^{0}$ with vertices $\{i, j, k\}$ such that for any  $y\neq x \in \{i, j, k\}$  we have
		\begin{equation}
	\nonumber	w (\mu_{x}(\mu_{y}(Q^{0})))> w(\mu_{y}(Q^{0})).
		\end{equation}
	And  $Q$ will be called \emph{pre-unbounded}  if there is a sequence of mutations $\mu$ such that $\mu(Q)$ is unbounded. An example of a pre-unbounded quiver is  the quiver in Part 3 of  Examples 2.8.
	\end{enumerate}
\end{defns}

\begin{thm}  Let $Q$ be a quiver. Then $Q$ is of finite mutation type, i.e., the mutation class $[Q]$ is a finite set if and only if  $w[Q]\leq 4$, for details see [3, 4].

\end{thm}

The main purpose of Lemma 3.6  is to characterise the notion of    unbounded and pre-unbounded  quivers introduced in Definition 3.4 Part 4. Which would  help  to develop a technique to use in predicting the future of a given quiver $Q$; that is whether $Q$  will develop an infinite mutation type quiver or not.
\begin{lem} For a quiver $Q$, the  following are equivalent
	
	\begin{enumerate}
		\item [(1)] $Q$ is of an  infinite mutation type.
		\item [(2)]  There is a rooted path $P(Q)$, such that the set of  weights  $\{\omega(Q'); Q' \in P(Q) \}$ is unbounded.
	   		\item [(3)]  There is a quiver $Q'\in [Q]$ which contains pre-unbounded  subquiver.
  \item [(4)] There is a quiver $Q'\in [Q]$ which contains a $3$-cycle that is not isosceles.
	\end{enumerate}
	
\end{lem}

\begin{proof}  ``$(1) \Rightarrow (2)$". Suppose that $\{w(Q'); Q' \in P(Q) \}$ is bounded for every path $P(Q)$ in $\mathbb{T}(Q)$. Then there is an a natural number  $\omega_{0}$ that is  an upper bound of the set of all  weights of quivers in $\mathbb{T}(Q)$, i.e., $w(\alpha) \leq \omega_{0}$,  for every edge $\alpha$ in any quiver in $P(Q)$.  Let $q(\omega)$ denote  the number of quivers of the same rank $n$ and with  weights less than or equal to $\omega$, i.e., $q(\omega)$ is the number of elements in the set $\{Q'; rk(Q)=n \ \text{and} \ w(Q)\leq \omega\}$. One can see that $q(\omega)$ is a finite number. Now we have  that  number of  diffrent quivers that show up in   $\mathbb{T}(Q)$  is less than or equal to the following finite sum of finite numbers
\[ \sum_{\omega\leq  \omega_{0}}q(\omega).\]
	
Therefore $[Q]$ is a finite set.

	$``(2) \Rightarrow (3)"$. Let $P(Q)$ be a rooted path  in the cluster  pattern $\mathbb{T}(Q)$ such that its set of  weights  $\{\omega(Q'), Q' \in P(Q) \}$ is unbounded. Then there is a quiver $Q' \in P(Q)$ which contains an edge $ \xymatrix{
		\cdot_{i} \ar[r]^{(b_{ij}, b_{ji})} & \cdot_{j} }$ where the weight $b_{ij}b_{ji}$ increases uncontrollably over the path $P(Q)$. The existence of  such edge  requires that it is an edge in   an unbounded $3$-cycle sub-quiver of $Q'$. Because bounded $3$-cycles are either periodic or breakable. Where periodic $3$-cycle  means applying mutations on it only its reverse the direction and breakable $3$-cycle means applying mutations removes one edge. Hence periodic $3$-cycle sub-quiver means the weights of all three edges stay with no change after applying any sequence of mutations. And breakable $3$-cycle sub-quiver means  the weight of at least one of its three edges becomes zero infinitely many times over any path.  Therefore the weights of all edges over any path are always bounded.  Which finishes the proof.

	$``(3) \Rightarrow (4)"$. Let  $Q'\in [Q]$ that contains pre-unbounded  subquiver. Then there is a sequence of mutations $\mu$ such that $\mu(Q')$ is an unbounded quiver. Therefore $\mu(Q')$ must contain a $3$-cycle subquiver where its weights increases  as we apply mutations. one can se that, applying  mutations on  different vertices of such $3$-cycle subquiver  produces another $3$-cycle that is  not isosceles.

   $``(4) \Rightarrow (1)"$. Let $Q'\in [Q]$ be a quiver that contains a $3$-cycle that is not isosceles. In such $3$-cycle, the minimum weights for two adjacent edges with distinct valuation is 2 and 3, which would correspond to valuations $(2, 1)$ and $(1, 3)$ respectively. Then the consistency condition implies that the third edge must be of valuation $(2, 3)$, which is of heavy weight.
	
\end{proof}

\begin{defn} A quiver $Q$ is called  \emph{vertex-to-vertex }$\sigma$-\emph{symmetric} if for every $i\in [1, n]$ there is a permutation $\sigma$ such that one of the following conditions is satisfied, either 
\begin{enumerate}
  \item we have $\mu_{i}(Q)=\pm \sigma (Q)$,  or
  \item there exists another vertex $j$ where  $\mu_{j}(\mu_{i}(Q))=\pm \sigma (Q)$.
\end{enumerate}
The vertex $j$ in case 2 will be called a \emph{counter vertex} of $i$ in $Q$.

 \end{defn}
 The following notation and remark will help organize and understand the second case of Definition 3.7  above.
    \begin{itemize}
      \item  For simplicity, we will use $v-v$, $\sigma$-symmetric for vertex-to-vertex symmetric and if $\tau=1$ we will omit the $\sigma$. We will use $\mu_{j, i}=\mu_{j}$ to refer to a mutation at a counter vertex $j$ of $i$ in $Q$ and $\mu_{\widetilde{i}}$ for an anonymous or undetermined  counter vertex for $i$.
      \item Inspired by Lemma 3.12  in [13], the second condition is equivalent to  the following: for every exchangeable vertex $i$ there is another vertex $k$ and a permutation $\tau$ such that  $\mu_{i}(Q)=\pm\tau(\mu_{k}(Q))$. Where for a permutation $\tau$ we have $\tau(\mu)=\mu_{\tau (i_{1})}\cdots \mu_{\tau (i_{n})}$.
    \end{itemize}
\begin{exams} The following are  $v-v$, $\sigma$-symmetric quivers

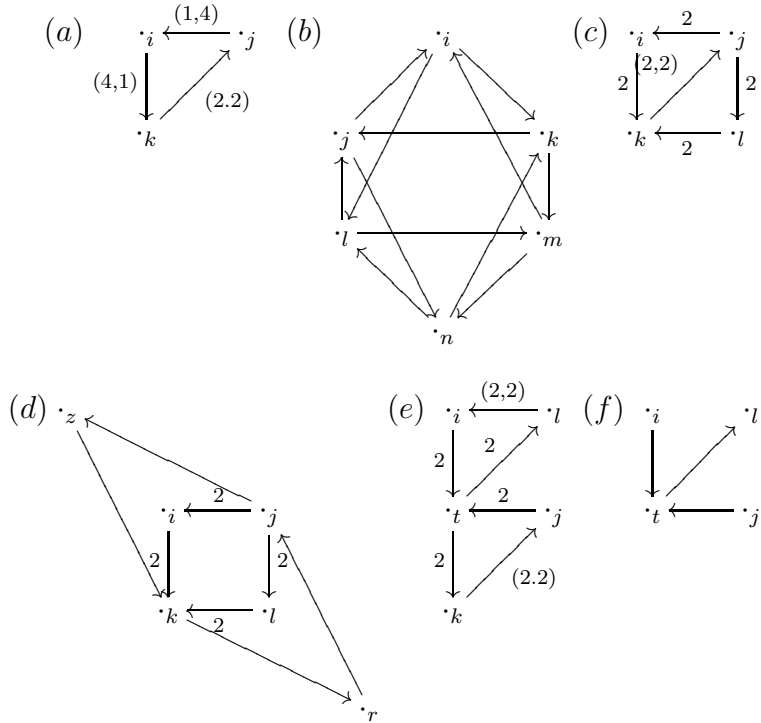
\begin{equation}
 \nonumber \  \ (a) \xymatrix{
	\cdot_{i} \ar[d]_{(4, 1)} & \ar[l]_{(1, 4)}\cdot_{j} \\
	\cdot_{k}  \ar[ur]_{(2. 2)}  }    \  \
(b) \xymatrix{ &\cdot_{i}\ar[ddl]\ar[dr]&\\
 \cdot_{j}\ar[ur]\ar[ddr]&&\cdot_{k}\ar[ll]\ar[d] \\
 \cdot_{l}\ar[u]\ar[rr]&&\cdot_{m}\ar[dl]\ar[uul]\\
 &\cdot_{n}\ar[ul]\ar[uur]& }
 (c) \xymatrix{
 \cdot_{i} \ar[d]_{2} & \ar[l]_{2}\cdot_{j} \ar[d]^{2}\\
  \cdot_{k}  \ar[ur]^{(2, 2)}& \ar[l]^{2}\cdot_{l} }
 \end{equation}

\begin{equation}\label{}
 \nonumber   (d) \xymatrix{ \cdot_{z}\ar[ddr]\\
 &\cdot_{i} \ar[d]_{2} & \ar[l]_{2}\cdot_{j}\ar[ull]\ar[d]^{2}\\
  &\cdot_{k} \ar[drr] & \ar[l]^{2}\cdot_{l} \\
   &&&\cdot_{r} \ar[uul]}
  (e) \xymatrix{\cdot_{i}\ar[d]_{2}&  \ar[l]_{(2,2)} \cdot_{l} \\
\cdot_{t}\ar[ur]^{2} \ar[d]_{2} & \ar[l]_{2}\cdot_{j} \\
	\cdot_{k}  \ar[ur]_{(2. 2)}}  \ \\
 \  (f) \ \xymatrix{\cdot_{i}\ar[d]& \cdot_{l} \\
\cdot_{t}\ar[ur]  & \ar[l] \cdot_{j}} \
\end{equation}
the weight 2 refers to the valuation $(2, 1)$ or $(1, 2)$ when it is appropriate.
For Example (a) we have $\mu_{x}(Q)=-Q, \ \forall x \in \{i, j, k\}$. For Example (b) we have $i$ is the counter vertex of $n$, $k$ is the counter vertex of $l$, and $m$ is the counter vertex of $j$. For Example (c),  We have $\mu_{x}=-Q, \ x \in \{j, k\}$ and $\mu_{j}$ is a counter vertex of  $\mu_{k}$ and vise versa. For Example (d),   $\mu_{j}$ is a counter vertex of  $\mu_{k}$, and $\mu_{z}$ is a counter vertex of  $\mu_{r}$ and vise versa respectively. For Example (e), we have  $\mu_{t}(Q)-(ji)(Q)$, $\mu_{l}(Q)=-(jk)(Q)$ and $\mu_{k}(Q)=-(li)(Q)$. For example (f), we have $\mu_{l}(Q)=-\mu_{j}(Q)$, $\mu_{j}(Q)=-(il)Q$ and $\mu_{i}(Q)=-(jl)Q$. The counter vertex of $i$ is $l$ where $\mu_{l}\mu(i)(Q)=(lt)(Q)$.

\end{exams}
The following proposition provides some properties of $v-v$, $\sigma$-symmetric quivers.
\begin{prop} If $Q$ is $v-v$, $\sigma$-symmetric, then we have

\begin{enumerate}

\item If $rk(Q)=3$, then $[Q]=\{Q, -Q\}$ or $Q$ is of $A_{3}$-type quiver.
\item If $rk(Q)=4$, then we have the following cases

\begin{enumerate}
  \item If $w(Q)=1$, then $Q$ is one of the following
  \begin{enumerate}
    \item $Q$ is of  $A_{4}$-type zigzag quiver.
    \item $Q$ is a one cycle of 4 vertices such that it is either an oriented cycle with at most 5 edges or a zigzag cycle with at most 4 edges.
    \item $Q$ is symmetric to the $\pm$ quiver $(f)$ in Example 3.8 above.
  \end{enumerate}

  \item If $w(Q)=4$, then $Q$ is either symmetric to the quiver $(c)$ in Example 3.8 above, or $Q$ is symmetric to  $\pm$ one of the following quivers

    \begin{equation}\label{}
     \nonumber         
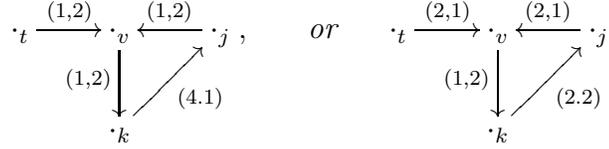
\xymatrix{
\cdot_{t} \ar[r]^{(1, 2)}&\cdot_{v} \ar[d]_{(1, 2)} & \ar[l]_{(1, 2)}\cdot_{j} \\
	&\cdot_{k}  \ar[ur]_{(4. 1)}}, \  \  \  \ \ \text{or} \ \ \  \nonumber   \  \xymatrix{
\cdot_{t} \ar[r]^{(2, 1)}&\cdot_{v} \ar[d]_{(1, 2)} & \ar[l]_{(2, 1)}\cdot_{j} \\
	&\cdot_{k}  \ar[ur]_{(2. 2)}} \  \  \  \
  \end{equation}
\end{enumerate}

  \item If $rk(Q)>4$, then $Q$ is fully cyclic such that every vertex in $Q$ is in an oriented cycle of at most 4 vertices.

  \item All $3$-cycles subquivers are either isosceles or equilateral triangles of maximum weight 4. Furthermore, if  $rk(Q)\leq 3$ then $Q$ is of finite mutation type and every $Q \in [Q']$ is also  $v-v$, $\sigma$-symmetric.

\end{enumerate}

\end{prop}
\begin{proof}
  \begin{enumerate}
  \item The case of $rk(Q)=3$ is obvious.

  \item Let $rk(Q)=4$. A good exercise is to show that all quivers described in Part 2 (a) and (b) are $v-v$, $\sigma$-symmetric.  Now, let $w(Q)=1$,  one can see that if $Q$ is acyclic but neither zigzag quiver nor  one oriented cycle then it is not $v-v$, $\sigma$-symmetric. If $w(Q)=1$ or $4$, and it  is not one cycle quiver but contains cyclic subquiver, then it must be a $3$-cycle attached to one leaf. If $Q$ is not symmetric to any of the quivers given in Part (b) above or quiver (c) in Example 3.8, then in such case there are very limited number of cases which would be easy to see that,  $Q$ will not be $v-v$, $\sigma$-symmetric.
    \item
     Let $rk(Q)\geq 5$. We have every acyclic quiver of rank $\geq 5$ is not $v-v$, $\sigma$-symmetric. Then  if the underlying graph of $Q$ contains  acyclic subquiver of at least five vertices then $Q$ can not be a $v-v$, $\sigma$-symmetric. The same argument can be made for any one-cycle quiver of 5 vertices or more.  Now, suppose that $Q$ is $v-v$, $\sigma$-symmetric which contains a leaf, say at a vertex $i$. Then since  $rk(Q)>2$ hence the vertex $i$ is connected to $Q$ through exactly one vertex say $j$ which  must be a part of a cycle. We will show that $\mu_{\tilde{j}}$ does not exist. Applying $\mu_{j}$ will create (at least one)  $3$-cycle with $i$ as a vertex in it and say $k$ as a third vertex. Hence the $Nhb._{\mu_{j}(Q)}(i)$ will contains $j$ and some other neighbor vertices to $j$. Any mutation in the direction of any vertex  in $Nhb._{\mu_{j}(Q)}(i)$ will either remove $[j,i]$ or reverse it. If $[j,i]$ was  reversed  then the new edge $[k,i]$ stays between this vertex and $i$ which means this mutation can not be $\mu_{\tilde{j}}$. The mutation on the directions of the other neighbor vertices of $j$ that are not in $Nhb._{\mu_{j}(Q)}(i)$ will not affect $[j, i]$ or $[k, i]$ which means there is  no counter vertices  of $j$ can be exist.
    \item
Suppose that there is one $3$-cycle that is not  isosceles,  equilateral or  of maximum weight 4. Then there is a $3$-cycle with three edges with three different weights. Hence it is of a maximum weight bigger than 4, i.e.,  this $3$-cycle will be unbounded, thanks to Lemma 3.6. Which means there is at least one  vertex with no counter vertex. Also, one can see that un-oriented $3$-cycles are in fact not $v-v$, $\sigma$-symmetric which means $Q$ does not have any unbounded subquivers. For the last part,  if $Q$ is $v-v$, $\sigma$-symmetric quiver  of  rank 3  then it would have a small mutation class, which is also easy to check.

      \end{enumerate}

\end{proof}

\begin{ques} Let $Q$ be $v-v$, $\sigma$-symmetric quiver.

 \begin{enumerate}
 \item What sufficient conditions on $rk(Q)$ and/or $w(Q)$ such that  $Q$ is of finite mutation type.
\item  What sufficient conditions on  $Q$ such that every quiver  $Q'$ in $[Q]$ is also $v-v$, $\sigma$-symmetric?

        \end{enumerate}
\end{ques}

We speculate that the answer for the second question to be a very small class of quivers. Inspired by Theorem 3.5, Lemmas 3.6 and Proposition 3.9, in the following  we will develop a method to classify the mutation classes based on their weights using what we call \emph{leading quivers}.

\begin{lem} Let $Q$ be a  quiver of  finite mutation type. Then the weight of $[Q]$ is determined as follows
\begin{enumerate}
  \item  $w[Q]=2$  if and only if $[Q]$ contains a quiver  that has exactly one  edge of weight $2$,   such that the sub-quiver    $Q_{\{i, j\}}$ is of $A$-type or $rk(Q)=2$. In other words, $[Q]$ contains a quiver of the type $B_{n},C_{n}$ or $F_{4}$.
  \item $w[Q]=3$ if and only if $[Q]$  contains the   quiver $ \xymatrix{\cdot_{i} \ar[r]^{(3, 1)} & \cdot_{j}}.$
  \item  $w[Q]=4$  if and only if one of the following cases, depending  on the rank of $Q$, is satisfied
   \begin{enumerate}
     \item [(i)] If $rk(Q)=3$  then  $[Q]$ contains  one of the quivers
      \begin{equation}
  \nonumber  
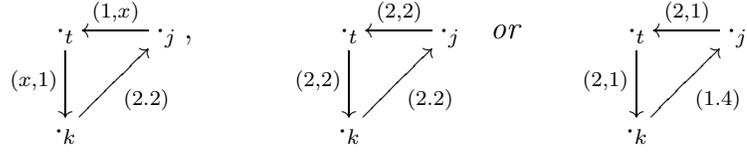
\xymatrix{
\cdot_{t} \ar[d]_{(x, 1)} & \ar[l]_{(1, x)}\cdot_{j} \\
	\cdot_{k}  \ar[ur]_{(2. 2)}}, \ \ \ \ \ \ \ \
 \xymatrix{
\cdot_{t} \ar[d]_{(2, 2)} & \ar[l]_{(2, 2)}\cdot_{j} \\
	\cdot_{k}  \ar[ur]_{(2. 2)}}  \ \ \text{or} \ \ \ \
\  \xymatrix{
\cdot_{t} \ar[d]_{(2, 1)} & \ar[l]_{(2, 1)}\cdot_{j} \\
	\cdot_{k}  \ar[ur]_{(1. 4)}}
 \end{equation}
 where $x=1, 2, 3$ or $4$.
     \item [(ii)] If $rk(Q)>3$ then $[Q]$ must satisfy the following criteria: Every quiver $Q'\in [Q]$  of weight 4  satisfies the following

       \begin{enumerate}
         \item [A.] Edges of weight  4 in $Q'$  appear in  a cyclic subquiver which is symmetric to  one of the following quivers

 \begin{equation}
\nonumber  (a)  \  \   \ Q_{a, x}: \  
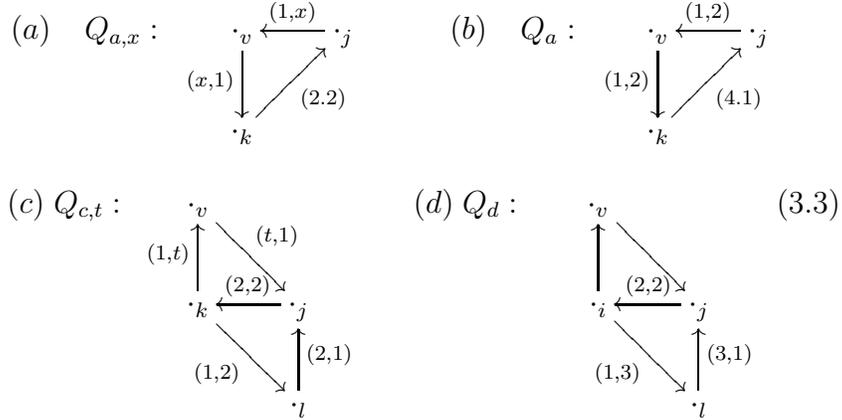
\xymatrix{
\cdot_{v} \ar[d]_{(x, 1)} & \ar[l]_{(1, x)}\cdot_{j} \\
	\cdot_{k}  \ar[ur]_{(2. 2)}} \  \  \  \ \ \ \  \
 (b)  \  \  \ Q_{a}: \  \xymatrix{
\cdot_{v} \ar[d]_{(1, 2)} & \ar[l]_{(1, 2)}\cdot_{j} \\
	\cdot_{k}  \ar[ur]_{(4. 1)}} \  \  \  \
 \end{equation}
 \begin{equation}\label{}
 (c) \  Q_{c, t}: \  \xymatrix{\cdot_{v}\ar[dr]^{(t, 1)}\\
 \cdot_{k}  \ar[u]^{(1, t)}\ar[dr]_{(1, 2)}& \cdot_{j}\ar[l]_{(2, 2)} \\
 & \ar[u]_{(2, 1)}\cdot_{l} }\\ \  \  \ \  \  \
 (d)  \  Q_{d}:  \  \  \  \ \  \ \xymatrix{\cdot_{v}\ar[dr]\\
 \cdot_{i}  \ar[u]\ar[dr]_{(1, 3)}& \cdot_{j}\ar[l]_{(2, 2)} \\
& \ar[u]_{(3, 1)}\cdot_{l} } \\
 \end{equation}

where $x=1, 2, 3$ or $4$  and  $t=1$ or $ 2$, such that edges of weight 4 are not connected outside their cycles. And any subquiver of the form $ \xymatrix{ \cdot \ar@{-}[r]^{z} &\cdot  \ar@{-}[r]^{z} &\cdot}, z=2$ or  $3$  appears in a quiver that is mutationally equivalent to one of the forms in (3.3).
 \item [B.]  $Q'$ will have more than one edge of weight 4  if it  is symmetric to  one of the following cases

      \begin{itemize}
        \item $Q'$ is formed from two or three copies of $Q_{a, 1}$  by coherently connecting them at $v$ such as in $X_{6}$ and $X_{7}$.  Or $Q$ is formed from $Q_{a, 2}$ and/or $Q_{a}$  in the following form

             \begin{equation}\label{}
              \  
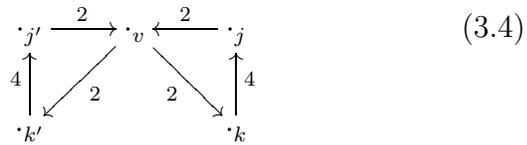
\xymatrix{
\cdot_{j'} \ar[r]^{2} &\cdot_{v}\ar[dl]^{2} \ar[dr]_{2} & \ar[l]_{2}\cdot_{j} \\
\cdot_{k'}\ar[u]^{4}&&\cdot_{k}  \ar[u]_{4}}
            \end{equation}

        \item  $Q'$ is formed from $Q_{a, 1}$, $Q_{a, 2}$ and/or $Q_{a}$  in one of the following forms

       \begin{enumerate}
\item [(a)]
            \begin{equation}\label{}
             \  \xymatrix{
\cdot_{j'}\ar[r]^{2}&\cdot_{v'} \ar[d]^{2}\ar@{-}[r] &\cdots\ar@{-}[r]&\cdot_{v} \ar[d]_{2} & \ar[l]_{2}\cdot_{j} \\
&\cdot_{k'}\ar[ul]^{4}&&\cdot_{k}  \ar[ur]_{4}},
            \end{equation}
           where the subquiver connecting $v$ and $v'$ is of $A$-type. Or $Q$ is the quiver
 \item [(b)]

            \begin{equation}\label{}
             \  \xymatrix{
\cdot_{j'}\ar[r]&\cdot_{v'} \ar[d]\cdot\ar@{-}[r]&\cdot_{v} \ar[d] & \ar[l]\cdot_{j} \\
&\cdot_{k'}\ar[ul]^{4}&\cdot_{k}  \ar[ur]_{4}}
            \end{equation}
        \end{enumerate}
      \end{itemize}
 Any additional subquiver of $Q'$ attached to a vertex in the quivers $Q_{a, x}, Q_{a},  Q_{b}$ or $Q_{c, t}$ will be referred to as a `` tail" and we will refer to $Q_{a, x}, Q_{a}, Q_{b}$ or $Q_{c, t}$ as a ``head".

     \item [C.] Tails of a weight-4 quiver $Q'\in [Q]$ satisfy the following
     \begin{itemize}
       \item If $Q'$ is one of the quivers $Q_{a, x}, x=1, 2, 3$, $Q_{a}$ or $Q_{c, 1}$ in (3.3)  then $Q'$ could have one simply-laced tail attached at the vertex $v$.

       \item  If $Q'$ is one of the quivers in (3.4), (3.5) or (3.6)  then it does not have any tails.
     \end{itemize}

       \end{enumerate}

   \end{enumerate}

\end{enumerate}

\end{lem}
 For simplicity, we will call the conditions given in Part 3 of Lemma 3.11, the ``\emph{weight-4 criteria}".
\begin{proof} \begin{itemize}
                \item  Parts 1,  2 and Part 3 case \emph{(i)} are obvious.
                \item Proof of Part 3 Case \emph{(ii)}. $``\Rightarrow"$.
\begin{itemize}
 \item [weight 4 edges.] If a weight $4$ edge is connected to another edge that is  not part of a triangle, i.e.,  $Q$ contains a subquiver of the form
 $\xymatrix{\cdot \ar[r]^{4} & \cdot\ar[r]&\cdot }$. Then a heavy weight edge can be easily created by applying  a  sequence of mutations. Which means all weight 4 edges must be completely within oriented cycles. In particular,  weight 4 edges are not  directly connected to tails (if any). Thanks to Lemma 3.6,  all triangles are isosceles then all possible   options  are the ones in (3.3).

  \item [More about tails.] Let $Q$ be of the type $Q_{a, 1}$  with a length 2 tail of weight $2$, i.e., a tail  formed of two edges, at least one of which is of weight 2. Then one can assume that $[Q]$ contains a quiver which has a subquiver that is symmetric to the quiver

\begin{equation}
 \  \  \xymatrix{
 &\cdot_{i} \ar[dr]& \cdot_{j}\ar[l]_{(2, 2)} \\
 \cdot_{v}& \ar[l]\cdot_{k}& \ar[l]^{(2, 1)}\ar[u]\cdot_{l} }.
\end{equation}
Applying the  sequence $\mu_{l}\mu_{k}\mu_{j}\mu_{i}\mu_{l} \mu_{k}$ will create a heavy weight edge. A similar sequence of mutations can be formed if the weight 2 edge was $[kv]$. Then a quiver  of finite  mutation type of  $Q_{a, 1}$ type can not have a tail of weight two and length two. A similar argument can be made if $Q$ is of the other two types $Q_{b}$ and $Q_{c, 1}$. Now, if $Q$ has a tail of weight $2$ connected to non-simply-laced edges. Then $[Q]$ contains a quiver that has a sub-quiver that is symmetric to the quiver

  \begin{equation}
 \  \  \xymatrix{\cdot_{j}\ar[dr]\\
 \cdot_{i}  \ar[u]\ar[dr]_{(1, 2)}& \cdot_{j}\ar[l]_{(2, 2)} \\
  \cdot_{k}& \ar[l]^{(2, 1)}\ar[u]_{(2, 1)}\cdot_{l} }.
\end{equation}
Applying $\mu_{i}\mu_{j}\mu_{l}$ will create a heavy weight edge. Then, if the head is a rectangular cycle and has a  tail that is connected to an edge of  weight two or more, then for $Q$ to be of finite mutation type the tail must be simply-laced sub-quiver.

One can see that if $Q$ has a tail of weight bigger than 2 then a sequence of mutation can be also  applied to create a heavy weight edge. Then tails are subquivers of weight at most $2$, for $Q$ to be of finite mutation type. And in the case of a weight 2 tail exists then it must be a one single edge that is connected  only to simply-laced edges.

Finally, one can show that if  any of the quivers in (3.4), (3.5) or (3.6) has a tail attached at $v$ or $v'$, then a heavy weight edge can be formed. Similarly, if any of the heads in (3.3)  has two tails where both are connected only at $v$, i.e., the two tails are  not part of a $3$-cycle, then one can come up with a sequence of mutations to produce a heavy weight edge.

\end{itemize}
                 $``\Leftarrow" $. In the following we will show that for a quiver $Q$, if $[Q]$  satisfies the weight 4 criteria then $w[Q]=4$. First, we will use mathematical induction on the length of mutations sequences to show that every quiver produced  from $Q$  satisfies the same criteria or has a reduced weight.  Then we will explain why  a heavy weight edge can not be created by applying any sequence of mutations on $Q$.
\begin{enumerate}
  \item [Initial step]

               Let $Q'=\mu_{j}(Q)$. First assume that $j$ is a vertex in a tail. Let $i \in Nhb.(j)$. If $i$ is connected to $j$ by a simply-laced edge such that $i$ is not a vertex in any cycle, then $i$ will stay connected to a simply-laced edge in $Q'$ and if a $3$-cycleto be formed due to applying $\mu_{j}$, it will be  isosceles or equilateral triangle.  If $i$ and $j$ are connected by a non-simply-laced edge, then $j$ will be a leaf or connected from the other side by a simply laced edge. In either cases $i$ will still satisfy the criteria in $Q'$ or the weight of the edge of weight  4  will be reduced.

               Now suppose that $i$ is a vertex in a $3$-cycle but $j$ is not, then same as the previous case, notice that no change will occur in the existing cycles but a new $3$-cycle could be formed which would be an isosceles triangle.

               Secondly, suppose that $j$ is a vertex in a head. If $j$ is facing the weight 4 edge, then the $3$-cycle direction will be reversed or the weight 4 edge will be removed or educed in $Q'$. The later possibility is due to the weights inequity, as the possible triangles are of weights $[(1, 2), (2, 1), (2, 2)], [(1, 2), (1, 2), (4, 1)]$ or $[(1, 3), (3, 1), (2, 2)]$. Notice that, based on the weight 4 criteria, we can't   have any weight $4$ edge that is not part a head.

               In the case of the weight 4 edge is removed then the quiver $Q'$ will have a weight less than 4, which moves it to one of the less weight classes.  If a $3$-cycle to be formed by applying $\mu_{j}$ then it will be again in an isosceles $3$-cycle of weight 2, since $j$ will be connected by a simply-laced edge to outside the $3$-cycle of weight 4.
               Therefore, $Q'$ still satisfy the weight 4 cetraria.
     \item [Induction step] Now assume that $\mu(Q)$ satisfies the weight 4 criteria for any sequence of mutations of length $k$. Hence from the previous step, one can see that a new weight 4 edge would not be formed unless if it is replacing a removed edge and isosceles  or  simply-laced equilateral triangles in $Q$ are not getting any bigger weight in $\mu_{j}(\mu (Q))$ for eavery $j\in [1, n]$. So, we have $\mu_{j}(\mu(Q))$  satisfies the weight 4 criteria. Then $w(\mu_{j}(\mu(Q)))\leq 4$, which finishes the proof.
    \end{enumerate}
\end{itemize}
\end{proof}

\begin{cor and defn}\begin{itemize}
\item Every mutation class of weight 2 can be produced from a quiver that has exactly one  edge $[ij]$ of weight 2, $\xymatrix{\cdot_{i}\ar[r]^{2}&\cdot_{j} } $, such that the subquiver    $Q_{\{i, j\}}$ is of $A$-type or trivial.

                      \item  The quivers appeared in Lemma 3.11 will be called ``leading quivers".
                    \end{itemize}

\end{cor and defn}

 \textbf{Notations} \begin{enumerate}
\item Let $\mu=\mu_{i_{1}}\cdots \mu_{i_{n}}$. We  write $\overleftarrow{\mu}:=\mu_{i_{n}}\cdots \mu_{i_{1}}$  and  $\{\mu\}=\{\mu_{i_{1}},\ldots, \mu_{i_{n}}\}$. 
\item Let $i, j\in [1, n]$ be two adjacent vertices, we denote the \emph{pentagon  sequences} of mutations $\mu_{i}\mu_{j}\mu_{i}\mu_{j}\mu_{i}$ by $\mu_{[i,j]}$ and when we write   $\mu_{\bar{i}}$ we mean a  sequence of mutations such that $\{\mu_{\bar{i}}\}$ does not contain $\mu_{i}$ except within a subsequence of mutations of the form $\mu_{[i,j]}$ or $\mu_{[j,i]}$. An example of  $\mu_{\overline{i}}$ is $\mu_{j_{1}}\mu_{j_{2}}\mu_{j_{3}}\mu_{j}\mu_{i}\mu_{j}\mu_{i}\mu_{j}\mu_{j_{4}}$, where $j\in Nhb.(i)$.

\end{enumerate}

The following proposition is inspired by  [1, Theorem 2.6] and [12, Proposition 3.12].
\begin{prop}  Let $Q$ be a  quiver such that $rk(Q)> 2$. Then  the following are equivalent
\begin{enumerate}

    \item [(1)] $\mathcal{S}_{n} \subset \mathcal{M}(Q)$;
    \item [(2)]  $\mathcal{S}_{n} \subset \mathcal{M}(Q')$, for every $Q' \in [Q]$;
    \item [(3)]  $Q$ is either

    \begin{itemize}
      \item   simply-laced quiver of finite mutation type, or
      \item  $Q$ is a $v-v$, $\sigma$-symmetric quiver of rank three.
    \end{itemize}

  \end{enumerate}
\end{prop}
\begin{proof}
\begin{enumerate}
\item  $(1)\Rightarrow (2)$. Let $\sigma (Q)=\mu_{\sigma, Q}(Q)$ and $Q' \in [Q]$ such that $Q'=\mu_{0}(Q)$ for some sequence of mutations $\mu_{0}$. Therefore, one can see that
\begin{equation}
\nonumber \sigma (\mu_{0}) \mu_{\sigma, Q} \overleftarrow{\mu}_{0}(Q')=\sigma (Q').
\end{equation}
\item $(2)\Rightarrow (1)$. Obvious.

\item $(1)\Leftrightarrow (3)$. $``\Leftarrow" $ If $Q$ is simply-laced, then  [1, Theorem 2.6] and [12, Proposition 3.12] provide the proof of this case. The case of  $Q$ is a $v-v$ symmetric quiver of rank 3 is an easy exercise, thanks to Proposition 3.9.

  $``\Rightarrow" $. We only need to prove the statement is true for one quiver in $[Q]$ thanks to Part (1) of the proposition. Let $Q$ be a quiver that is not simply-laced nor $v-v$, $\sigma$-symmetric of rank 3. We will divide the proof into two cases.
\begin{enumerate}
\item First case, if $Q$ is of finite mutation type quiver. Without loss of generality, assume that $Q$ contains a leaf,  $\xymatrix{\cdot_{i} \ar[r]^{(b_{ij}, b_{ji})} & \cdot_{j}}$, with weight $\vert b_{ij}b_{ji}\vert \in \{2, 3, 4 \}$, connected from the vertex $i$ with a simply laced edge to the rest of the quiver $Q$. We will show, by mathematical induction, that any sequence of mutations $\mu$ will keep the vertex $j$ as a leaf or turn it to be a vertex in a cyclic triangle. Which would mean the permutation $(ij)\notin M(Q)$. The mathematical induction is on the number of single mutations in $\{\mu \}$. An easy exercise is to show that the statement is true if $\vert\{\mu\}\vert=1,2$ or $3$. Assume that every sequence of mutations  $\mu$ such that $\vert\{\mu\}\vert =k$ will send $j$ to either a leaf or a vertex in a cyclic triangle. Again an easy exercise is show that for any $k \in [1, n]$ we have $\mu_{k}\mu$ will send $i$ to a cyclic $3$-cycleor a leaf.
\item Second case, let $Q$ be of an infinite mutation type.   Lemma 3.6 guarantees the existence of a path $P(Q)$  in $\mathbb{T}_{n}(Q)$  that contains  a quiver $Q'$ with a cyclic  triangular sub-quiver $[i, j, t]$ with an edge $\xymatrix{\cdot_{i}  \ar[r] ^{(b_{ij}, b_{ji})} &\cdot_{j} }$ such that applying any sequence of mutations containing $\mu_{t}, \mu_{i}$ or $\mu_{j}$ on $Q'$ will result in increasing  the weight of one of edges of the $3$-cycle$[i,j,k]$. Thus the production of a $3$-cycle$[\sigma (i), \sigma(j), \sigma(k)]$ with same weights as of the $3$-cycle$[i,j,k]$ is impossible for any permutation $\sigma$ which means that the permutation $(ij)$ can not be an element of $\mathcal{M}(Q')$ hence $(ij)\notin \mathcal{M}(Q)$. Then  the symmetric group $\mathcal{S}_{n}$ is not a subset of $\mathcal{M}(Q)$.
\end{enumerate}

\end{enumerate}

\end{proof}

\begin{defn}

  Let $i$ be a vertex in $Q$ such that all edges connected to it are simply-laced. Then we say that the vertex $i$ has  \emph{a simply-laced avenue} if it  satisfies the following

  \begin{enumerate}
    \item $i$ is not directly connected to any pre-unbounded subquiver;
    \item $i$ is connected to another vertex that is at least two edges away from the nearest non simply-laced edge, if any.
  \end{enumerate}

Example 3.3 Part 2 quiver (3.1) vertex 2 is an example of a vertex with a simply-laced avenue.

\end{defn}

\begin{prop}  Let $(X, Q)$ be a seed  and  $i$ be a vertex  with simply-laced avenue. Then  there is a sequence of mutations $\mu_{\bar{i}}$ and a permutation $\tau$ such that $\mu_{\bar{i}}\mu_{i}(Q)=\pm\tau(Q)$. In particular we have
\begin{enumerate}
  \item the statement is true for simply-laced  quivers of finite mutation type and $v-v$ symmetric quivers.
  \item  let $x_{i}$ be the initial cluster variable at the vertex  $i$ then $\mu_{i}(x_{i})$ is a cluster variable in  the seed $(\mu_{\bar{i}}\mu_{i}(X), \pm\tau(Q))$.
\end{enumerate}

\end{prop}

\begin{proof} Let $i$ be a vertex with a simply-laced avenue. Then, all the vertices in $Nhb.(i)$ are all connected to $i$ through simply-laced edges and the nearest vertex that is connected to a non simply-laced edge, if any, is connected to $i$ via a subquiver (path) of length at  least 2. Without loss of generality, we can assume that this vertex exists and call it $j$, otherwise all other vertices connected to $i$ would be leafs, which will not change the proof. Now, let $k$ be the middle vertex between $i$ and $j$. Apply $\mu_{k}\mu_{[ik]}\mu_{i}$, which will produce  $\pm\tau(Q)$ for some permutation $\tau$, also one can see $\mu_{i}(x_{i})$ is a cluster variable in the seed $\mu_{k}\mu_{[ik]}\mu_{i}(X, Q)$, thanks to [12, Proposition 3.12]. Therefore, $\mu_{\overline{i}}=\mu_{k}\mu_{[ik]}$.
If $i$ is a vertex in a  $v-v$ symmetric quiver $Q$, then $\mu_{\overline{i}}=\mu_{l}$, where $l$ is a $v-v$ counter vertex to $i$ in $Q$.
\end{proof}

The technique detailed in the proof of Proposition 3.15 will be called the \emph{simply-laced avenue technique}.

\begin{cor}
  For every seed  $(X, Q)$,  such that $ \mathcal{S}_{n}\subset \mathcal{M}(Q)$,   the following two properties are satisfied

  \begin{enumerate}
    \item The simply-laced avenue technique is applicable at every vertex in $Q$.
    \item  We have  \begin{equation}
 \nonumber    \mathcal{B}(Q)=\mathcal{A}(Q).
  \end{equation}
  \end{enumerate}

  \end{cor}
  \begin{proof}
   \begin{enumerate}
     \item Proposition 3.13 implies that $Q$ is either of finite  mutation type or a rank three $v-v$ symmetric. If $Q$ is simply-laced of finite type then every vertex has a simply-laced avenue, then Proposition 3.15 is satisfied. The case of $Q$ is $v-v$ symmetric of rank 3 is obvious.
     \item Notice that since $\mathcal{S}_{n}\subset \mathcal{M}(Q)$, then $ \mathcal{S}_{n}\subset \mathcal{M}(Q')$ for every quiver $Q' \in [Q]$, thanks to Proposition 3.13. Hence every vertex in every $Q'\in [Q]$ has a simply-laced avenue, thanks to Part (1).

     Let $y$ be a cluster variable rooted at $k$ produced by a shortest mutation sequence  $\mu_{y}=\mu_{k}\mu_{j_{t}}\cdots \mu_{j_{1}}$ of length $t+1$. Part (1) guarantees that there is a sequence of mutations $\mu_{\overline{k}}$ such that $\mu_{\overline{k}}(\mu_{j_{t}}\cdots \mu_{j_{1}} (Q)) =\sigma_{t+1} (\mu_{j_{t}}\cdots \mu_{j_{1}} (Q))$ where $\sigma_{t+1}$ is a permutation, thanks to proposition 3.15. Since $Q$ is simply-laced or a rank three $v-v$ symmetric quiver then one can see that $y$ is a cluster variable in the seed  $\mu_{\overline{k}}(\mu_{j_{t}}\cdots \mu_{j_{1}} (X, Q)) =(Y_{t+1}, \sigma_{t+1} (\mu_{j_{t}}\cdots \mu_{j_{1}} (Q)))$. Now, use the assumption that  $\mathcal{S}_{n}\subset M$ and switch $y$ to the vertex $\sigma_{t+1}(j_{t})$. Apply the previous step on $\sigma_{t+1}(j_{t})$ so we get to the seed $(Y_{t}, \sigma_{t+1} \sigma _{t}(\mu_{j_{t-1}}\cdots \mu_{j_{1}} (Q)))$, where $Y_{t}=\mu_{\overline{\sigma_{0}(j_{t}))}}(Y_{t+1})$ which contains $y$.  Keep applying this process $t-1$ times on $\mu_{t-1}, \ldots, \mu_{1}$ respectively, we get the seed $(Y_{1}, \sigma_{t+1}\sigma_{t} \dots \sigma_{1}(Q))$, where $y\in Y_{1}=\mu_{\overline{y}}(X)$ for some mutation sequence $\mu_{\overline{k}}$. Which finishes the proof.
   \end{enumerate}

  \end{proof}

\begin{cor} If $Q$ is of finite  mutation  simply-laced quiver  then $\mathcal{B}(Q)=\mathcal{A}(Q)$.
\end{cor}
\begin{proof} Proposition 3.13  implies that  $ \mathcal{S}_{n}\subset \mathcal{M}(Q)$. Hence  Corollary 3.16 guarantees that  $\mathcal{B}(Q)=\mathcal{A}(Q)$.
\end{proof}

The main aim of the rest of the paper is to  show that the statement of  Corollary 3.17 is valid for a much bigger  class of quivers.

\begin{lem} If $Q$ satisfies one of the following two cases
\begin{enumerate}
  \item Either $Q$ is one of the  leading quivers from Lemma 3.11  with no subquiver that is symmetric to any of the quivers in (3.2), i.e., $Q$ has no rigid vertices, or
  \item Every quiver in $[Q]$ is a $v-v$, $\sigma$-symmetric quiver.
\end{enumerate}  Then for every $Q'\in [Q]$ and for every $i\in [1, n]$ there is a mutation sequence $\mu_{\bar{i}}$ and a permutation $\tau$ such that $\mu_{\bar{i}}\mu_{i}(Q')=\tau(Q)$.
\end{lem}
\begin{proof}
\begin{description}
  \item[(1).]

  Fix a leading quiver $Q$.  If the simply-laced avenue technique is applicable on $i$ then use it to prove the statement is true for $\mu_{i}$.  The case of  $i$ does not have a simply-laced avenue is divided into three subcases based on the weight of $[Q]$.
\begin{enumerate}
  \item Let $[Q]$ be of weight 2. Let $Q$ be a quiver with acyclic underlying graph  and one single edge of weight 2, to be the leading quiver. Also, one can assume that every $Q'\in [Q]$ contains at most three edges of weight 2, and every edge of weight 2 appears in a  triangular sub-quiver that is symmetric to the following quiver
      \begin{equation}\label{}
     \xymatrix{
\cdot_{t} \ar[d]_{(2, 1)} & \ar[l]_{(1, 2)}\cdot_{j} \\
	\cdot_{k}  \ar[ur]}.
      \end{equation}

 Let $Q'\in [Q]$, we divide the set $\{\mu_{i}(Q'); i \in [1, n]\}$ into two classes
      \begin{enumerate}
        \item First class contains   $\mu_{i}(Q')$ if it  is acyclic or if it contains only simply-laced cycles. In such case, it is straightforward to find a sequence of mutation $\mu_{\bar{i}}$ and a permutation $\tau$ such that $\mu_{\bar{i}}\mu_{i}(Q')=\tau(Q)$ as the differences between $Q$ and $\mu_{i}(Q')$ would be possibly only directions of some arrows after breaking all the cycles, if any.
        \item Second class contains all $\mu_{i}(Q'), i\in [1, n]$ which has a $3$-cycle that is symmetric to the quiver in (3.9).  If $i$ is not connected to any edge of weight 2 then it is an easy case to do.  Without loss of generality, we can assume that applying $\mu_{i}$ to $Q'$ creates a cycle that is symmetric to (3.9). We have three possible cases for the subquiver containing $i$ in $\mu_{i}(Q')$:

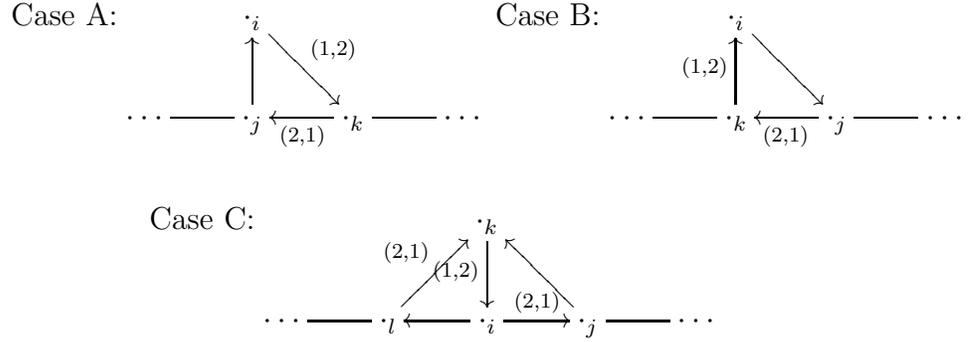
          \begin{equation}\label{}
     \nonumber \text{Case A:} \xymatrix{ &\cdot_{i}\ar[dr]^{(1, 2)}&\\
      \cdots & \cdot_{j}\ar[u]\ar@{-}[l]&\cdot_{k}\ar[l]^{(2, 1)}\ar@{-}[l]& \ar@{-}[l]\cdots } \  \text{Case B:} \xymatrix{ &\cdot_{i}\ar[dr]&\\
      \cdots & \cdot_{k}\ar[u]^{(1, 2)}\ar@{-}[l]&\cdot_{j}\ar[l]^{(2, 1)}\ar@{-}[l]& \ar@{-}[l]\cdots }
          \end{equation}

            \begin{equation}\label{}
      \nonumber  \text{Case C:} \xymatrix{& &\cdot_{k}\ar[d]_{(1, 2)}&&\\
      \cdots& \ar@{-}[l]\cdot_{l} \ar[ur]^{(2, 1)}& \cdot_{i}\ar[l]\ar[r]^{(2, 1)}&\cdot_{j}\ar[ul]\ar@{-}[l]& \ar@{-}[l]\cdots }
          \end{equation}

        In all three  cases the first step is to straighten the simply-laced part of the quiver $\mu_{i}(Q')$, i.e., break all the 3-cycles.
\begin{enumerate}
\item [Case A:] Apply mutation at $j$ which will break the triangle, and a possible $3$-cyclewith $i$ as a vertex will be formed or $i$ will be  a leaf. If $i$ becomes a leaf, apply $\mu_{k}$ so applying $\mu_{j}$, if needed, would not recreate the original triangle. In such case $i$ can be relocated  to a spot where it has a simply laced avenue.  Finally, move the weight 2 edge to the same spot as it was in $Q$ and adjust the directions to create $\tau(Q)$, for some permutation $\tau$.

\item [Case B:] In this case breaking the $3$-cycleby applying $\mu_{j}$ will make $i$ a leaf and if a  $3$-cycleis formed, we can break it without using $\mu_{i}$ which turn the quiver to be acyclic.  Then adjust the arrows and move the weight 2 edge, if needed, to get $\tau(Q)$ for some permutation $\tau$.

\item  [Case C:] In this case apply $\mu_{l}\mu_{k}$, which will create a simply-laced quiver with one edge of weight 2. Now apply necessarily mutations to move the weight 2 edge to a similar spot as in $Q$ and finally adjust the direction so that we get $\tau(Q)$ for some permutation $\tau$.
  \end{enumerate}

      \end{enumerate}

  \item Let $w[Q]=3$. Then  $rk(Q)$ is 2 which is Obvious.
  \item Let $w[Q]=4$. From   Lemma 3.11, we have if $rk(Q)\geq3$, then any edge  of weight 4 will appear in a cycle that is symmetric to  one of the following

 \begin{equation}
\nonumber (a)   Q_{ax}: \  
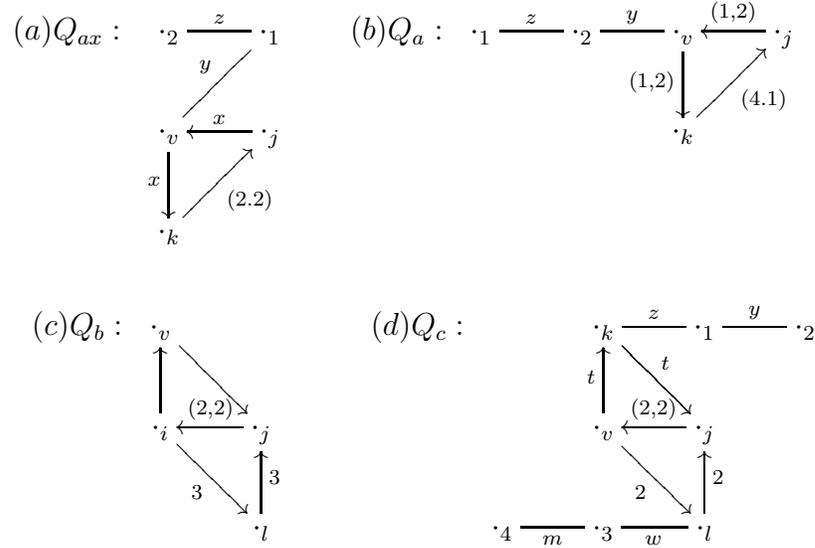
\xymatrix{\cdot_{2}\ar@{-}[r]^{z}&\cdot_{1} \\
\cdot_{v}\ar@{-}[ru]^{y} \ar[d]_{x} & \ar[l]_{x}\cdot_{j} \\
	\cdot_{k}  \ar[ur]_{(2. 2)}}  \  \  \  \ \  \
(b)  Q_{a}: \  \xymatrix{
\cdot_{1}\ar@{-}[r]^{z}&\cdot_{2}\ar@{-}[r]^{y}&\cdot_{v} \ar[d]_{(1, 2)} & \ar[l]_{(1, 2)}\cdot_{j} \\
	&&\cdot_{k}  \ar[ur]_{(4. 1)}} \  \  \  \
\end{equation}

\begin{equation}
\nonumber (c)   Q_{b}:  \ \xymatrix{\cdot_{v}\ar[dr]\\
 \cdot_{i}  \ar[u]\ar[dr]_{3}& \cdot_{j}\ar[l]_{(2, 2)} \\
& \ar[u]_{3}\cdot_{l} } \  \ \ \ \ \ \ \
(d)  Q_{c}: \  \xymatrix{&\cdot_{k}\ar@{-}[dr]\ar[dr]^{t}&\ar@{-}[l]_{z}\cdot_{1}&\cdot_{2}\ar@{-}[l]_{y}\\
 &\cdot_{v}  \ar[u]^{t}\ar[dr]_{2}& \cdot_{j}\ar[l]_{(2, 2)} \\
\cdot_{4} \ar@{-}[r]_{m}& \ar@{-}[r]_{w}\cdot_{3}& \ar[u]_{2}\cdot_{l} }
\end{equation}
where $x\in \{1, 2,  3\}, y\in\{0,1\}, z=\{0, 1\}$  and  $t\in \{1, 2\}$.

\begin{enumerate}
  \item \textbf{Case $Q_{a1}$}. This class of quivers has the following leading quivers  $X_{6}, X_{7}, E^{(1, 1)}_{6}, E^{(1, 1)}_{7}$ and $E^{(1, 1)}_{8}$, $Q_{a1}$ or quiver symmetric to (3.6) with $(y, z) =(1, 1)$ or $(1, 0)$. There are  main features that are satisfied by all of these quivers:
       \begin{enumerate}
         \item If $Q'$ is in a mutation class of any of these quivers, then  it is either a simply-laced quiver or all its  non simply-laced edges  occur in  cycles that are symmetric to $3$-cycle in $Q_{a1}$;
         \item The mutations at any vertex in an edge of weight 4 will not change the underlying graph of $Q$ as it will either
          produce $-Q'$ or there is, easy to find, sequence of mutations $\mu_{\bar{i}}$ such that $\mu_{\bar{i}}(\mu_{i}(Q'))=-Q'$.
       \end{enumerate}

       These two features make applying simply-laced avenue technique available for vertices that are not part of weight 4 edges and for the vertices that are in weight 4 edges reversing the directions of the quiver is always possible.
\item

     \textbf{Case} $Q_{a2}$. In this case our leading quiver is one of the following two cases

     \begin{enumerate}
      \item $z=0$ and $y=2$. In this case $Q$ is $v-v$ symmetric of 4 vertices and very small mutation class which is easy to verify.
      \item If $Q$ is symmetric to (3.4) is obvious.

           \item $y=1$ and $z=1, 2$ or $Q$ has a subquiver that is symmetric to (3.5). In this case $Q$ has  rigid vertices  at $v$ and $1$ respectively.

     \end{enumerate}

\item
      \textbf{Case} $Q_{a3}$. The leading quiver is one of the following cases

      \begin{enumerate}
      \item z=0 and y=1. Which is $v-v$ symmetric with $\mu_{\bar{v}}=\mu_{j}, \mu_{\bar{j}}=\mu_{k}$ and $\mu_{\bar{1}}=\mu_{j}$.

           \item y=1 and z=1. Which is an infinite quiver. You can apply $\mu_{v}\mu_{1}\mu_{2}\mu_{k}\mu_{j}\mu_{1}\mu_{v}$ to create a heavy weight edge.
     \end{enumerate}
\item
      \textbf{Case} $Q_{a}$ with $y\in\{0,1, 2\}, z=\{0, 1\}$. If $y=0  \ \text {or} \ 1$, then there is a rigid vertex at $v$. If $y=2, z=0$, then it desponds on the  valuation  of the edge $[2v]$. It is either mutation equivalent to a $Q_{a2}$ type quiver or it is $v-v$ symmetric of 4 vertices and very small mutation class which is again easy to verify..

  \item  \textbf{Case $Q_{b}$}. In this case, $Q$ is $v-v$ symmetric of 4 vertices and very small mutation class.

  \item
  \begin{enumerate}
    \item

  \textbf{Case $Q_{c}$  with $t=1$}. Then we have the following subcases

  \begin{enumerate}
    \item   $Q_{c}$ with  $z=y=w=m=0$.  Then the vertex $l$ is rigid.
    \item   $Q_{c}$ with  $y=0$ and $z=w=1$, i.e., $Q=F_{4}^{(*,*)}$. In this case $k$ has a simply-laced avenue which will be used when needed.
    if $Q'=Q$ and $i=l$. Then $\mu_{\bar{i}}=\mu_{3}\mu_{1}\mu_{j}\mu_{v}\mu_{k}\mu_{j}$. If $i=j$ and $Q'=\mu_{k}(Q)$, then  $\mu_{\bar{j}}=\mu_{v}$ and then reverse the directions of the arrows of the tails if needed.  If  $Q'=\mu_{k}(Q)$ or $Q'=\mu_{k}\mu_{j}\mu_{v}\mu_{j}\mu_{k}(Q)$. Then apply $\mu_{j}\mu_{v}$ and $\mu_{j}$ respectively. Now, we divide the rest of quivers in $[Q]$ into two types based on their weight. One can see that, for $Q'\in [Q]$ we have
     $w(Q')$ will be altered only if $Q'$ is obtained from $Q$ by applying a sequence of mutations contains $\mu_{k}$ and/or $\mu_{l}$. Let  $Q'=\mu(Q)$. If $\mu_{l}$ does not divide $\mu$, then the proof is straightforward. Now assume that $\mu$ is devisable by $\mu_{l}$. We have the following three cases
    \begin{itemize}
      \item Let $\mu=\mu_{\ast}\mu_{l}$ such that $\mu_{l}$ does not divide $\mu_{\ast}$, i.e., $w(Q')=2$. Then apply $\mu_{\overline{l}}\overleftarrow{\mu_{\ast}}$.
      \item Let $\mu=\mu_{\ast}\mu_{l}$ such that $\mu_{\ast}$ is devisable by $\mu_{l}$ and  $w(\mu_{\ast}(Q))=4$ then the only thing needed is to adjust the edges directions to obtain $\tau(Q)$ for some permutation $\tau$.
      \item Let $\mu=\mu_{\ast}\mu_{l}$ such that $\mu_{\ast}$ is devisable by $\mu_{l}$ and  $w(\mu_{\ast}Q)<4$ then whether the weight of the edge $[jk]$ is one or zero, one can produce $\tau(\mu_{l}(Q))$ without using $\mu_{l}$ noting the symmetry between the vertices $j$ and $k$ and $k$ has a simply-laced avenue.
    \end{itemize}

    \item Let $y=z=0$ and $w=m=1$, i.e., $Q$ is $F_{4}^{(*, +)}$. Then $\mu_{\bar{i}}=\mu_{j}\mu_{k}\mu_{v}\mu_{4}\mu_{3}$. If $i=j$ and $Q'=\mu_{k}(Q)$, then  $\mu_{\bar{j}}=\mu_{v}$ and then reverse the directions of the arrows of the tails, if needed. If $i$ is any other  vertex, the proof is pretty similar to Case B, noting the symmetry between $v$ and $k$.

    \item Let $Q$ contains the following  as subquiver
    \begin{equation}\label{}
 \nonumber 
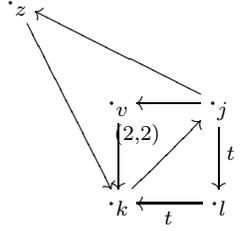
\xymatrix{ \cdot_{z}\ar[ddr]\\
 &\cdot_{v} \ar[d] & \ar[l]\cdot_{j}\ar[ull] \ar[d]^{t}\\
  &\cdot_{k}  \ar[ur]^{(2, 2)}& \ar[l]^{t}\cdot_{l} }.
 \end{equation}
  \end{enumerate}
Where $t=1,2$.

For $t=1$, this subquiver appears in $E_{6}^{(1, 1)}, E_{7}^{(1, 1)}$ and $E_{6}^{(1, 1)}$. In such case when the weight 4 edge will be reduced to 1 by applying mutations at $l, v$ or $z$ each vertex would have a simply-lace avenue, which make the proof straightforward. When $t=2$,  one can see that the symmetry between the vertices  $j$ and $k$ and $v$ from one side and $l$ and $z$ from the other side respectively which makes it easy to find the sequences $\mu_{\bar{i}}$ for $i\in Q_{1}$.

 \item

 \textbf{Case $Q_{c}$  with $t=2$}. We have two main cases
  \begin{enumerate}
    \item  $z=w=y=m=0$.  then $Q_{c}$ is $v-v$ symmetric of 4 vertices and very small mutation class which is easy to verify.
    \item If $t=2$ and $Q$ has tails of maximum weight equals one each, and $Q'=Q$ then use  $\mu_{\bar{k}}=\mu_{l}$ and $\mu_{\bar{j}}=\mu_{t}$ and then reveres the arrows of any tails, if any. All other cases of $Q'\in [Q]$ are similar.
  \end{enumerate}

\end{enumerate}
\end{enumerate}
\end{enumerate}
 \item[(2).] Let $Q'\in [Q]$ and $i\in [1, n]$. Then $Q'$ is $v-v$, $\sigma$-symmetric, i.e.,  there are $j, \mu_{i_{1}},\ldots,  \mu_{i_{k}}\in [1, n]$ such that $\mu_{j}\mu_{i}(Q')=\sigma (Q')=\sigma(\mu_{i_{1}}\cdots \mu_{i_{k}}(Q))$. Therefore
 $\mu_{j}\mu_{i}(Q')=(\mu_{{\sigma(i_1)}}\cdots \mu_{{\sigma(i_k)}}(\sigma (Q)))$. Suppose that $\sigma (i_t)=i$, for some $t\in \{i_{\sigma (i_1)}, \ldots, i_{\sigma(i_k)}\}$. Let $\mu_{\widetilde{\sigma(i_{t})}}$ be a counter vertex of  $\mu_{\sigma(i_{t})}$ for the quiver $\mu_{{\sigma(i_{t+1})}}\cdots \mu_{{\sigma(i_k)}}(\sigma (Q))$.
 Then there is a permutation $\tau$ such that

 \begin{eqnarray}
 \nonumber  \mu_{\sigma (i_t)}\mu_{\sigma(i_{t+1})}\cdots \mu_{\sigma(i_{k})}(Q) &=& \tau(\mu_{\widetilde{\sigma(i_{t})}} \mu_{\sigma(i_{t+1})}\cdots \mu_{\sigma(i_{k})}(\sigma(Q)) \\
\nonumber    &=& \mu_{\tau(\widetilde{\sigma(i_{t})})} \mu_{\tau(\sigma(t+1))}\cdots \mu_{\tau(\sigma(i_{k}))}(\tau(\sigma(Q))).
 \end{eqnarray}
 Apply the same technique with every time we have $\mu_{\sigma(i_{x})}=\mu_{i}$ for any $x\in \{i_{1}, \ldots, i_{k}\}$. Therefore, there is $\mu=\mu_{j_{1}}\cdots\mu_{j_{k}}$ and permutation $\eta$ such that $\mu_{j}\mu_{i}(Q')=\mu(\eta(Q))$ where $i \notin \{\mu\}$. Hence $\overleftarrow{\mu}\mu_{j}\mu_{i}(Q')=\eta(Q)$. Then $\mu_{\overline{i}}=\overleftarrow{\mu}\mu_{j}$, which finishes the proof of this case.
\end{description}
\end{proof}

\begin{thm} Let $\Sigma=(X, Q)$ be a seed. Then $ \mathcal{A}(Q)$  is symmetric if and only if    $|Q|<\infty$   with no rigid vertices, i.e., $Q$ does not contain any subquiver that is $\sigma$-similar to any of the quivers in (3.2).

\end{thm}
  \begin{proof} $ ``\Leftarrow"$\\
  Let $z$ be a cluster variable in $\mathcal{X}_{Q}$ rooted at the initial cluster variable $x_{l}$. Then there is a quiver $Q'\in [Q]$ such that $z$ is a cluster variable in the seed $(Y, \mu_{l}(Q'))$. Then Lemma 3.18  guarantees the existences of a sequence of mutations $\mu_{\overline{l}}$ and  permutation $\tau$ such that $\mu_{\overline{l}}\mu_{l}(Q')=\tau(Q)$. Therefore $z$ belongs to the cluster of the seed $(\mu_{\overline{l}}(Y), \tau(Q))$, which means $z\in \mathcal{B}(Q)$.

  $ ``\Rightarrow"$\\
   Let $Q$ be a quiver with an infinite mutation class. Then there is a seed $(Z, Q')$ with a quiver $Q'\in [Q]$ such that there  are two vertices $i$ and $j$ in  $Q'_{1}$,  the set of vertices of $Q'$, such that
   \begin{itemize}
   \item The vertices $i$ and $j$ form a $3$-cycle with a third vertex say $k$;
   \item Applying any sequence of mutations of the form $\mu^{m}_{ijk}, m\in \mathbb {N}$ on $Q'$ where $\mu_{ijk}=\mu_{i}\mu_{j}\mu_{k}$  will produce a  quiver containing the $3$-cycle $i, j, k$ such that the weight $m_{ij}$ of the edge  $\cdot_{i}\longrightarrow \cdot_{j}$ is directly proportional to $m$. So increasing $m$ will result in increasing of the weight $m_{ij}$ with no upper bound;
   \item There is $m_{0}$ such that the set of quivers  $\{\mu^{m_{0}+t}_{ijk}(Q'); t\geq 1\}\in [Q]$ has no two  symmetric quivers  for  any permutation $\sigma$;
   \item Consider the set $X_{j}=\{x_{j}(t)=\mu^{m_{0}+t}_{ijk}(z_{j}), t\geq 1\}$  where $z_{j} \in Z$ is the cluster variable rooted at $j$, so $X_{j}$ is the set of cluster variables rooted at $j$ and produced from $z_{j}$. The set $X_{j}$ contains no repeated cluster variables.
  \end{itemize}
  In the following we will show that $x_{j}(1)$ does not belong to any seed that is symmetric to the initial seed.  Notice that this is the first time $j$ connects with $k$ so all cluster variables with  shortest mutations sequence rooted at $j$ do not have any cluster variable rooted at $k$ in their cluster formula. Also, every $x_{j}(t), t\geq 2$ contains cluster variables rooted at $k$ but with higher exponent so their denominators vectors would contain different $k^{\text{th}}$-component. Finally since the $3$-cycle contains $i, j$ and $k$ are not breakable nor periodic. Then non of the quivers in $\mathbb{T}(Q)$ is similar to the quiver contains is attached to a cluster containing $x_{1}(1)$.

  One can see that there is no sequence of mutations that can change the triangular cyclic to remove it or make the whole quiver similar to $Q$. i.e., every sequence of mutations produce a new quiver not similar to any.

  Finally, let $Q$ be one of the quivers in (3.2), i.e., it has a rigid vertex, say at $i$. After finding the whole mutation class and tracking down all the cluster variables, one can see that some of the cluster variables rooted at $i$ are not symmetric.
  \end{proof}

\begin{cor}For a quiver $Q$, if every quiver in $[Q]$ is a $v-v$, $\sigma$-similar quiver, then $\mathcal{A}(Q)$ is a symmetric cluster algebra.
\end{cor}
\begin{proof} Thanks to Lemma 3.18 and the Proof of Theorem 3.19.
\end{proof}

\begin{cor}[Properties of symmetric mutations algebras] Let $\widetilde{Q}$ be a quiver
\begin{enumerate}

\item The symmetric mutation algebra $ \mathcal{B}(Q)$ is the $\mathbb{Z}$-algebra generated by the set  $\bigcup\limits_{\grave{Q}\leq Q}\mathcal{X}_{\mathcal{B}(Q')}$.

   \item
   The following are equivalent
   \begin{enumerate}
     \item  The symmetric mutation algebra $\mathcal{B}(Q)$ is a sub-cluster algebra of $\mathcal{A}(Q)$, i.e., $\mathcal{B}(Q)=\mathcal{A}(Q')$ for a subquiver $Q'$;
     \item $Q'$ is the largest connected finite type subquiver of $Q$ that does not  contain any rigid vertices or any pre-unbounded triangles of $Q$.
     \item $\widetilde{Q}$ has a unique decomposition $\widetilde{Q}=\breve{Q}\bigodot \hat{Q}$, where $\check{Q}$ is a finite mutation subquiver  and $\hat{Q}$ is an infinite and/or with rigid vertices or trivial subquiver.
   \end{enumerate}

\end{enumerate}

\end{cor}

\begin{proof}
\begin{enumerate}
\item  Note: the infinite subquiver and the quivers with rigid vertices do not contribute with any cluster variables in $\mathcal{X}_{\mathcal{B}(Q)}$.  Let $y\in B(Q)$ with an initial cluster variable rooted at $x_{k}$. Then $x_{k}\in Q'_{0}$ for some finite mutation type subquiver $Q'$ where the vertex $k$ is not rigid, thanks to Theorem 3.19. Now, let $y\in \mathcal{B}(Q')$, for some subquiver $Q'$ of finite mutation type and no rigid vertices. Then there exist a sequence of mutations $\mu$ such that $y\in \mu (X, Q')=(\mu(X), \tau(Q')$ for some permutation $\tau$. Hence, $y\in (\mu(X), \tau(Q))$ by extending $\tau$ with identities for the vertices of $Q\backslash Q'$.

\item  $ (a)\Rightarrow (b)$. Assume $\mathcal{B}(Q)$ be a subcluster algebra. Then there is a subseed $Q'$ such that  $\mathcal{B}(Q)=\mathcal{A}(Q')$. If $Q$ contains another symmetric subquiver $Q''$ that is not subquiver of $Q'$ then some of the symmetric cluster variables of $\mathcal{B}(Q)$ are in $\mathcal{A}(Q'')$ which are not in $\mathcal{A}(Q)$. Then $Q'$ must be the largest connected subquiver of $Q$.

    $ (b)\Rightarrow (c)$ and $ (c)\Rightarrow (a)$ are obvious.

\end{enumerate}
\end{proof}

The exceptional types $E_{6}, E_{7}, E_{8}, E^{(1)}_{6}, E^{(1)}_{7}, E^{(1)}_{8}, E^{(1,1)}_{7}, E^{(1,1)}_{8}, X_{6}$ and $X_{7}$.

\begin{itemize}
  \item[$E_{6}$:] \begin{equation}\label{}
  \nonumber    
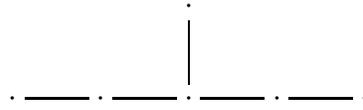
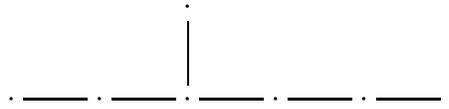
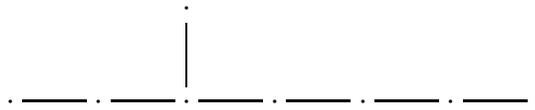
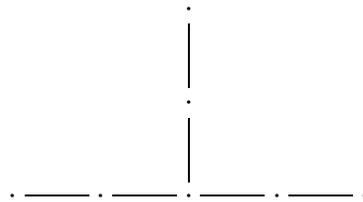
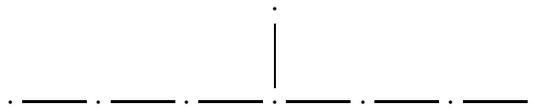
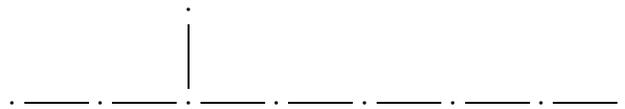
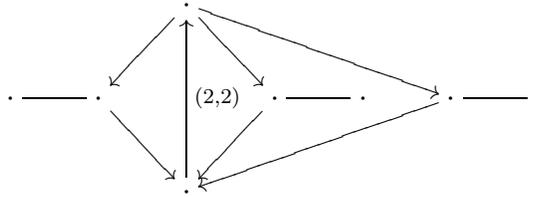
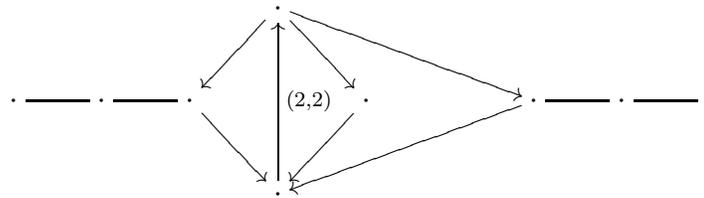
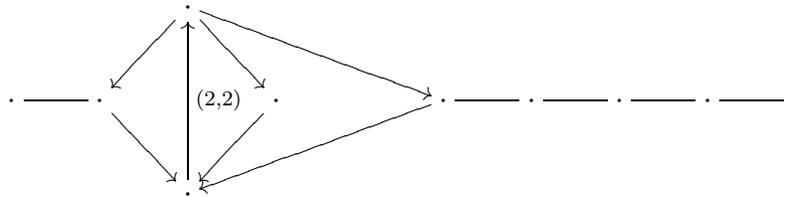
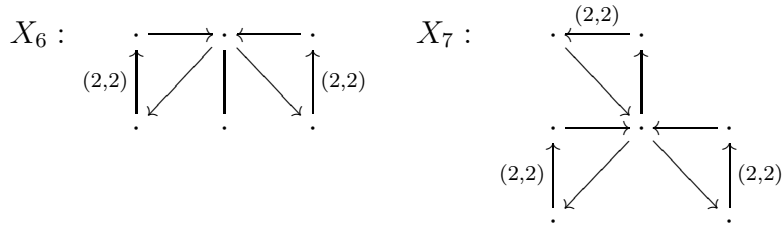
\xymatrix{&&\cdot\ar@{-}[d]&& \\
\cdot\ar@{-}[r]&\cdot\ar@{-}[r]&\cdot\ar@{-}[r]&\cdot\ar@{-}[r]&\cdot}
\end{equation}
  \item [$E_{7}$:] \begin{equation}\label{}
  \nonumber    \xymatrix{&&\cdot\ar@{-}[d]&& \\
\cdot\ar@{-}[r]&\cdot\ar@{-}[r]&\cdot\ar@{-}[r]&\cdot\ar@{-}[r]&\cdot &\cdot\ar@{-}[l]}
\end{equation}
  \item  [$E_{8}$:] \begin{equation}\label{}
  \nonumber    \xymatrix{&&\cdot\ar@{-}[d]&& \\
\cdot\ar@{-}[r]&\cdot\ar@{-}[r]&\cdot\ar@{-}[r]&\cdot\ar@{-}[r]&\cdot &\cdot\ar@{-}[l]&\cdot\ar@{-}[l]}
\end{equation}
  \item [$E^{(1)}_{6}$:] \begin{equation}\label{}
  \nonumber    \xymatrix{&&\cdot\ar@{-}[d]&&\\
  &&\cdot\ar@{-}[d]&& \\
\cdot\ar@{-}[r]&\cdot\ar@{-}[r]&\cdot\ar@{-}[r]&\cdot\ar@{-}[r]&\cdot}
\end{equation}
   \item [$E^{(1)}_{7}$:] \begin{equation}\label{}
  \nonumber    \xymatrix{&&&\cdot\ar@{-}[d]&& \\
\cdot\ar@{-}[r]&\cdot\ar@{-}[r]&\cdot\ar@{-}[r]&\cdot\ar@{-}[r]&\cdot\ar@{-}[r]&\cdot &\cdot\ar@{-}[l]}
\end{equation}

   \item  [$E^{(1)}_{8}$:] \begin{equation}\label{}
  \nonumber    \xymatrix{&&\cdot\ar@{-}[d]&& \\
\cdot\ar@{-}[r]&\cdot\ar@{-}[r]&\cdot\ar@{-}[r]&\cdot\ar@{-}[r]&\cdot \ar@{-}[r]&\cdot\ar@{-}[r]&\cdot\ar@{-}[r]&\cdot}
\end{equation}
  \item  [$E^{(1, 1)}_{6}$:] \begin{equation}\label{}
 \nonumber \xymatrix{ && \cdot\ar[dl]\ar[dr]\ar[drrr] &&&&\\
 \cdot\ar@{-}[r] &\cdot\ar[dr] &&\cdot\ar[dl]\ar@{-}[r]&\cdot&\cdot\ar[dlll]\ar@{-}[r]&\cdot\\
 &&\cdot\ar[uu]_{(2,2)}&&&&}
 \end{equation}

  \item  [$E^{(1, 1)}_{7}$:] \begin{equation}\label{}
 \nonumber \xymatrix{ &&& \cdot\ar[dl]\ar[dr]\ar[drrr] &&&&&\\
 \cdot\ar@{-}[r]&\cdot\ar@{-}[r] &\cdot\ar[dr] &&\cdot\ar[dl]&&\cdot\ar[dlll]\ar@{-}[r]&\cdot&\ar@{-}[l]\cdot\\
 &&&\cdot\ar[uu]_{(2,2)}&&&&&}
 \end{equation}

  \item [$E^{(1, 1)}_{8}$:] \begin{equation}\label{}
 \nonumber \xymatrix{ && \cdot\ar[dl]\ar[dr]\ar[drrr] &&&&&&&\\
 \cdot\ar@{-}[r] &\cdot\ar[dr] &&\cdot\ar[dl]&&\cdot\ar[dlll]\ar@{-}[r]&\cdot&\ar@{-}[l]\cdot &\ar@{-}[l]\cdot &\ar@{-}[l]\cdot\\
 &&\cdot\ar[uu]_{(2,2)}&&&&&&&}
 \end{equation}

  \item  [ ] \begin{equation}\label{}
               \nonumber X_{6}: \xymatrix{
\cdot \ar[r] &\cdot\ar[dl] \ar[dr] & \ar[l]\cdot \\
\cdot\ar[u]^{(2,2)}&\cdot\ar@{-}[u]&\cdot  \ar[u]_{(2,2)}}
            \ \ \ \  X_{7}:
              \  \xymatrix{\cdot\ar[dr]&\cdot\ar[l]_{(2,2)}\\
\cdot \ar[r] &\cdot\ar[dl] \ar[dr]\ar[u] & \ar[l]\cdot \\
\cdot\ar[u]^{(2,2)}&&\cdot  \ar[u]_{(2,2)}}
            \end{equation}
\end{itemize}

\subsection*{Acknowledgments} I would like to express my appreciation to the anonymous reviewer for all the valuable remarks and corrections. I would like to thank Fan Qin for sharing his  lecture notes with me which was for  the international workshop in cluster algebras and related topics in Tianjin, China, Summer 2017. I also would like to thank Fang Li and  Zongzhu Lin for very valuable discussions regarding the topic of this paper. Finally, I would like to thank the Math Department of Maquette University for the hospitality during writing this article.


\end{document}